\documentclass[pdftex]{amsart}
\usepackage{amsmath, amssymb, url}
\usepackage{type1cm}
\usepackage{tikz, tikz-cd}
\usetikzlibrary{matrix,arrows,shapes,calc}

\usepackage[top=30truemm,bottom=30truemm,left=35truemm,right=35truemm]{geometry}

\newtheorem{defi}{Definition}[section]
\newtheorem{thm}[defi]{Theorem}
\newtheorem{prop}[defi]{Proposition}
\newtheorem{lem}[defi]{Lemma}
\newtheorem{cor}[defi]{Corollary}
\newtheorem{eg}[defi]{Example}
\newtheorem{rem}[defi]{Remark}
\newtheorem{conj}[defi]{Conjecture}

\newtheorem{assum}[defi]{Assumption}
\newtheorem{que}[defi]{Question}

\DeclareMathOperator{\add}{add}

\DeclareMathOperator{\Auteq}{Auteq}
\DeclareMathOperator{\CC}{\mathbb{C}}

\DeclareMathOperator{\Cok}{Cok}
\DeclareMathOperator{\Cone}{Cone}
\DeclareMathOperator{\D}{D^b}

\DeclareMathOperator{\EE}{\mathcal{E}}
\DeclareMathOperator{\End}{End}
\DeclareMathOperator{\Ext}{Ext}

\DeclareMathOperator{\FM}{FM}

\DeclareMathOperator{\gldim}{gldim}
\DeclareMathOperator{\Gr}{Gr}

\DeclareMathOperator{\Hom}{Hom}
\DeclareMathOperator{\Ima}{Im}
\DeclareMathOperator{\IW}{\Phi}

\DeclareMathOperator{\lext}{\mathcal{E}\!\mathit{xt}}
\newcommand{\id}{\mathrm{id}}

\DeclareMathOperator{\Ker}{Ker}

\DeclareMathOperator{\LL}{\mathcal{L}}
\DeclareMathOperator{\LLL}{\mathbb{L}}
\DeclareMathOperator{\LGr}{LGr}
\DeclareMathOperator{\map}{\longrightarrow}
\DeclareMathOperator{\modu}{mod}

\DeclareMathOperator{\Per}{Per}
\DeclareMathOperator{\PP}{\mathbb{P}}
\DeclareMathOperator{\Pic}{Pic}

\DeclareMathOperator{\Qcoh}{Qcoh}

\DeclareMathOperator{\RRR}{\mathbb{R}}

\DeclareMathOperator{\RGamma}{R\Gamma}
\DeclareMathOperator{\RHom}{RHom}
\DeclareMathOperator{\Seg}{Seg}
\newcommand{\sh}{\mathcal{F}}

\DeclareMathOperator{\Spec}{Spec}
\DeclareMathOperator{\Sub}{\mathcal{S}}
\DeclareMathOperator{\ST}{T}

\DeclareMathOperator{\Sym}{Sym}
\DeclareMathOperator{\Tilt}{\mathcal{T}}
\DeclareMathOperator{\Tot}{Tot}

\DeclareMathOperator{\TU}{TU}
\DeclareMathOperator{\UT}{UT}
\newcommand{\stsh}{\mathcal{O}}

\newcommand{\RR}{\mathrm{R}}

\title[on derived equivalence for Abuaf flop]{on derived equivalence for Abuaf flop: \\ mutation of non-commutative crepant resolutions and spherical twists}
\author{wahei hara}
\address{The Mathematics and Statistics Building, University of Glasgow, University Place, Glasgow, G12 8QQ, UK.}
\email{wahei.hara@glasgow.ac.uk}
\date{}

\begin{document}

\begin{abstract}
In \cite{Seg16}, Segal constructed a derived equivalence for an interesting 5-fold flop that was provided by Abuaf.
The aim of this article is to add some results for the derived equivalence for Abuaf's flop.
Concretely, we study the equivalence for Abuaf's flop by using Toda-Uehara's tilting bundles and Iyama-Wemyss's mutation functors.
In addition, we observe a ``flop-flop=twist" result and a ``multi-mutation=twist" result for Abuaf's flop.
\end{abstract}

\maketitle

\tableofcontents

\section{Introduction}

\subsection{Motivation}

In \cite{Seg16}, Segal studied an interesting flop provided by Abuaf, which we call the \textit{Abuaf flop}.
Let $V$ be a four dimensional symplectic vector space and $\LGr(V)$  the Lagrangian Grassmannian.
Let $Y$ be a total space of a rank 2 bundle $\Sub(-1)$ on $\LGr(V)$, where $\Sub$ is the rank 2 subbundle and $\stsh_{\LGr(V)}(-1) := \bigwedge^2 \Sub$.
Then, $Y$ is a local Calabi-Yau $5$-fold.
On the other hand, let $\PP(V)$ be the projective space and put $\LL := \stsh_{\PP(V)}(-1)$.
Then using the symplectic form on $V$ gives an injective bundle map $\LL \hookrightarrow \LL^{\bot}$.
Let $Y'$ be the total space of a bundle $(\LL^{\bot}/\LL) \otimes \LL^2$,
then $Y'$ is another local Calabi-Yau $5$-fold with an isomorphism $H^0(Y, \stsh_Y) \simeq H^0(Y', \stsh_{Y'}) =: R$ of $\CC$-algebras.
Put $X := \Spec R$.
Abuaf observed that the correspondence $Y \rightarrow X \leftarrow Y'$ gives an example of $5$-dimensional flops.
This flop has the nice feature that the contracting loci on either side are not isomorphic.
In \cite{Li17}, Li proved that a simple flop of dimension at most five is one of the following
\begin{enumerate}
\item[(1)] a (locally trivial deformation of) standard flop,
\item[(2)] a (locally trivial deformation of) Mukai flop,
\item[(3)] the Abuaf flop.
\end{enumerate}
Standard flops and Mukai flops are well-studied.
Thus it is important to study the Abuaf flop from the point of view of Li's classification.

Based on the famous conjecture by Bondal, Orlov, and Kawamata, it is expected that $Y$ and $Y'$ are derived equivalent.
Segal proved that this expectation is true.
The method of his proof is as follows.
He constructed tilting bundles $\Tilt_{\mathrm{S}}$ and $\Tilt'_{\mathrm{S}}$ on $Y$ and $Y'$ respectively,
and proved that there is an isomorphism
\[ \End_Y(\Tilt_{\mathrm{S}}) \simeq \End_{Y'}(\Tilt'_{\mathrm{S}}). \]
Then, by using a basic theorem for tilting objects, we have a derived equivalence
\[ \Seg' : \D(Y') \xrightarrow{\sim} \D(Y). \]

On the other hand, in \cite{TU10}, Toda and Uehara established a method to construct a tilting bundle under some assumptions (Assumption \ref{assumption TU1} and Assumption \ref{assumption TU2}).
The difficulties to use Toda-Uehara's method are the following:
\begin{enumerate}
\item[(a)] There are few examples known to satisfy their assumptions.
\item[(b)] Since Toda-Uehara's construction consists of complicated inductive step, it is difficult to find an explicit description of the resulting tilting bundle in general.
\end{enumerate}
However, in the case of the Abuaf flop, it is possible show the following.

\begin{thm}[see Section \ref{subsec: TU assump}]
$Y$ and $Y'$ satisfy Toda-Uehara's assumptions.
\end{thm}

Hence the construction by Toda and Uehara gives new tilting bundles $\Tilt_{\mathrm{T}}$ and $\Tilt'_{\mathrm{T}}$ on $Y$ and $Y'$, respectively.
Moreover, fortunately, it is possible to compute the resulting tilting bundles explicitly in this case.
Using this explicit description of the tilting bundle shows that there is another tilting bundle $\Tilt_{\mathrm{U}}$ on $Y$ that satisfies
\[ \End_Y(\Tilt_{\mathrm{U}}) \simeq \End_{Y'}(\Tilt_{\mathrm{T}}'). \]
Therefore applying the basic theorem for tilting objects again provides a new derived equivalence
\[ \TU' : \D(Y') \to \D(Y). \]
Note that a tilting bundle constructed by using Toda-Uehara's method can be regarded as a canonical one, 
because it provides a projective generator of a perverse heart of the derived category.
Thus, it is quite natural to ask the following questions.

\begin{que} \label{intro Q} \rm
\begin{enumerate}
\item[(1)] What is the relation among three tilting bundles on $Y$, $\Tilt_{\mathrm{S}}$, $\Tilt_{\mathrm{T}}$, and $\Tilt_{\mathrm{U}}$?
\item[(2)] What is the relation between two tilting bundles on $Y'$, $\Tilt'_{\mathrm{S}}$ and $\Tilt'_{\mathrm{T}}$?
\item[(3)] What is the relation between two equivalences $\Seg'$ and $\TU'$?
\end{enumerate}
\end{que}

The aim of this article is to answer these questions.

\subsection{NCCRs and Iyama-Wemyss's mutations}
Put
\begin{align*}
\Lambda_{\mathrm{S}} &:= \End_Y(\Tilt_{\mathrm{S}}) = \End_{Y'}(\Tilt'_{\mathrm{S}}), \\
\Lambda_{\mathrm{T}} &:= \End_Y(\Tilt_{\mathrm{T}}), \\
\Lambda_{\mathrm{U}} &:= \End_Y(\Tilt_{\mathrm{U}}) = \End_{Y'}(\Tilt'_{\mathrm{T}}).
\end{align*}
Then these algebras are \textit{non-commutative crepant resolutions} (=NCCRs) of $X = \Spec R$.
The notion of NCCR was first introduced by Van den Bergh as a non-commutative analog of crepant resolutions.
An NCCR of a Gorenstein domain $R$ is defined as the endomorphism ring $\Lambda := \End_R(M)$ of a reflexive $R$-module $M$ such that
$\Lambda$ is Cohen-Macaulay as an $R$-module and its global dimension is finite.
As in the commutative case, a Gorenstein domain $R$ may have many different NCCRs.
One way to compare different NCCRs is \textit{Iyama-Wemyss's mutation} (= IW mutation).

Let $A$ be a $d$-singular Calabi-Yau algebra and $M$ an $A$-module whose endomorphism ring $\End_A(M)$ is an NCCR of $A$.
Let $N \in \add M$ and consider a right $(\add N^*)$-approximation of $M^*$
\[ a : N_0^* \to M^* \]
(see Definition \ref{def approx}).
Then IW mutation of $M$ at $N$ is defined as $\mu_N(M) := N \oplus \Ker(a)^*$.
In \cite{IW14}, Iyama and Wemyss proved that the endomorphism ring of $\mu_N(M)$ is also an NCCR of $A$ and there is a derived equivalence
\[ \Phi_N : \D(\modu \End_A(M)) \xrightarrow{\sim} \D(\modu \End_A(\mu_N(M))) \]
(see Theorem \ref{IW mutation equiv} for more detail).

In many cases, it is observed that important NCCRs are connected by multiple IW mutations.
For example, Nakajima proved that, in the case of three dimensional Gorenstein toric singularities associated with reflexive polygons, all \textit{splitting} NCCRs are connected by repeating IW mutations \cite{Nak16}.
In addition, the author studied IW mutations of certain NCCRs of the minimal nilpotent orbit closure of type A \cite{H17a}.
Also in the case of the Abuaf flop, we can show the following.

\begin{thm}[= Theorem \ref{mutation1}, Theorem \ref{mutation2}]
The above three NCCRs $\Lambda_{\mathrm{S}}$, $\Lambda_{\mathrm{T}}$, and $\Lambda_{\mathrm{U}}$ are connected by multiple IW mutations.
\end{thm}

This result provides an answer to Question \ref{intro Q} (1) and (2).
We prove this theorem by relating IW mutations with \textit{mutations of full exceptional collections} on $\D(\LGr(V))$
(see Section \ref{sect exc colle}).

\subsection{Flop-Flop=Twist result}
Recall that $Y$ and $Y'$ are $5$-dimensional local Calabi-Yau varieties.
It is known that the derived category of a Calabi-Yau variety normally admits an interesting autoequivalence called a \textit{spherical twist}
(see Section \ref{susect: def sph}).
Spherical twists arise naturally in mathematical string theory and homological mirror symmetry.

On the other hand, it is widely observed that spherical twists also appear in the context of birational geometry.
Namely, in many places, it is observed that a spherical twist arises as a difference of two derived equivalences associated to a flop  \cite{ADM15, BB15, Ca12, H17a, DW16, DW15, To07}.
We call this phenomenon ``flop-flop=twist".
The following theorem gives the first example of this phenomenon in the case of the Abuaf flop.

\begin{thm}[= Theorem \ref{thm flop-flop}, Theorem \ref{thm flop-flop2}]
\begin{enumerate}
\item[(1)] Let us consider a spherical twist $\ST_{\Sub[2]}$ around a $1$-term complex $\Sub[2] = \Sub|_{\LGr}[2]$ on the zero section $\LGr \subset Y$.
Then, we have a functor isomorphism
\[ \Seg' \circ \TU'^{-1} \simeq \ST_{\Sub[2]} \in \Auteq(\D(Y)). \]
\item[(2)] An autoequivalence $\TU'^{-1} \circ \Seg'$ of $\D(Y')$ is isomorphic to a spherical twist $\ST_{\stsh_{\PP}(-3)}$ associated to a sheaf $\stsh_{\PP}(-3)$ on the zero-section $\PP \subset Y'$:
\[ \TU'^{-1} \circ \Seg' \simeq \ST_{\stsh_{\PP}(-3)} \in \Auteq(\D(Y')). \]
\end{enumerate}
\end{thm}

To prove the first statement of the theorem, we provide an explicit description of a Fourier-Mukai kernel of an equivalence $\TU'$.
Let $\widetilde{Y}$ be a blowing-up of $Y$ along the zero section $\LGr = \LGr(V)$.
Then, the exceptional divisor $E$ of $\widetilde{Y}$ is isomorphic to $\PP_{\LGr}(\Sub(-1))$.
Thus we can embed $E$ into the product $\LGr(V) \times \PP(V)$ via an injective bundle map $\Sub(-1) \subset V \otimes_{\CC} \stsh_{\LGr}(-1)$.
Set $\widehat{Y} := \widetilde{Y} \cup_E (\LGr(V) \times \PP(V))$.
We prove the following.

\begin{thm}[= Theorem \ref{FM ker}]
The Fourier-Mukai kernel of the equivalence $\TU'$ is given by the structure sheaf of $\widehat{Y}$.
\end{thm}

Note that $\widehat{Y} = Y \times_X Y'$.
This is very close to the case of Mukai flops \cite{Ka02, Na03}.

\subsection{Multi-mutation=twist result.}
We also study a spherical twist from the point of view of NCCRs.
Namely, we can understand a spherical twist as a composition of IW mutations in the following way.
Let us consider a bundle on $Y$
\[ \Tilt_{\mathrm{U}, 1} := \stsh_Y(-1) \oplus \stsh_Y \oplus  \stsh_Y(1) \oplus \Sub(1). \]
We can show that this bundle is also a tilting bundle on $Y$.
Put
\begin{align*}
M &:= H^0(Y, \Tilt_{\mathrm{U}, 1}), \\
W' &:= H^0(Y, \stsh_Y \oplus  \stsh_Y(1) \oplus \Sub(1)), ~ \text{and} \\
\Lambda_{\mathrm{U}, 1} &:= \End_Y(\Tilt_{\mathrm{U}, 1}) \simeq \End_R(M).
\end{align*}
We show that there is an isomorphism of $R$-modules
\[ \mu_{W'}(\mu_{W'}(\mu_{W'}(\mu_{W'}(M)))) \simeq M \]
(Proposition \ref{mutation3}).
Furthermore, using Iyama-Wemyss's theorem gives an autoequivalence of $\D(\modu \Lambda_{\mathrm{U}, 1})$
\[ \nu_{W'} := \Phi_{W'} \circ \Phi_{W'} \circ \Phi_{W'} \circ \Phi_{W'} \in \Auteq(\D(\modu \Lambda_{\mathrm{U}, 1})). \]
This autoequivalence corresponds to a spherical twist in the following sense:

\begin{thm}[= Theorem \ref{multimutationtwist}]
The autoequivalence $\nu_{W'}$ of $\D(\modu \Lambda_{\mathrm{U}, 1})$ corresponds to a spherical twist
\[ \ST_{\stsh_{\LGr}(-1)} \in \Auteq(\D(Y)) \]
under the identification $\RHom_Y(\Tilt_{\mathrm{U}, 1}, -) : \D(Y) \xrightarrow{\sim} \D(\modu \Lambda_{\mathrm{U}, 1})$.
\end{thm}

Donovan and Wemyss proved that, in the case of $3$-fold flops, a composition of \textbf{two} IW mutation functors corresponds to a spherical-like twist \cite{DW16}.
In the case of Mukai flops, the author observed that a composition of \textbf{many} IW mutation functors corresponds to a P-twist \cite{H17a}.
The theorem above provides a new example of this phenomenon for the Abuaf flop.
%In Appendix \ref{appendix c}, we prove a ``multi-mutation=twist" result for the toric NCCR of a cyclic quotient singularity as another instance for this principle.

\subsection{Plan of the article}
In Section \ref{sect: prelim}, we provide some basic definitions and theorems we use in later sections.
In Section \ref{sect: mutation}, we give an explicit description of the tilting bundle obtained by Toda-Uehara's construction.
In addition, we show that NCCRs obtained as the endomorphism rings of Toda-Uehara's or Segal's tilting bundle are connected by repeating IW mutations.
In Section \ref{sect: twist}, we prove ``flop-flop=twist" results and a "multi-mutation=twist" result for the Abuaf flop and provide an explicit description of the Fourier-Mukai kernel of the functor $\TU'$.
%In Appendix \ref{section BBW},  we provide some explanation of representation theory and demonstrate how to use the Borel-Bott-Weil theorem for readers who are not familiar with it.
In Section \ref{sect exc colle}, we explain the definition of exceptional collections and its mutation.
As an application of them, we explain how to find a resolution of a sheaf from an exceptional collection.
%In Appendix \ref{appendix c}, we study IW mutations of the toric NCCR of a cyclic quotient singularity.

\subsection{Notations. } In this paper, we always work over the complex number field $\mathbb{C}$. Moreover, we adopt the following notations.

\begin{enumerate}
\item[$\bullet$] $V = \mathbb{C}^4$ : $4$-dimensional symplectic vector space.
\item[$\bullet$] $\mathbb{P}(V) := V \setminus \{0\} / \mathbb{C}^{\times}$ : projectivization of a vector space $V$.
\item[$\bullet$] $\LGr(V)$ : the Lagrangian Grassmannian of $V$.
\item[$\bullet$] $\Tot(\EE) := \Spec_X \Sym_X \EE^*$ : the total space of a vector bundle $\EE$.
\item[$\bullet$] $\modu(A)$ : the category of finitely generated right $A$-modules.
\item[$\bullet$] $\add(M)$ : the additive closure of $M$.
\item[$\bullet$] $\D(\mathcal{A})$ : the (bounded) derived category of an abelian category $\mathcal{A}$.
\item[$\bullet$] $\D(X) := \D(\mathrm{coh}(X))$ : the derived category of coherent sheaves on a variety $X$.
\item[$\bullet$] $\FM_{\mathcal{P}}$, $\FM_{\mathcal{P}}^{X \to Y}$ : A Fourier-Mukai functor from $\D(X)$ to $\D(Y)$ whose kernel is $\mathcal{P} \in \D(X \times Y)$.
\item[$\bullet$] $\ST_{\EE}$ : the spherical twist around a spherical object $\EE$.
\item[$\bullet$] $\mu_N(M)$ : the left (Iyama-Wemyss) mutation of $M$ at $N$.
\item[$\bullet$] $\Phi_N : \D(\modu \End_R(M)) \to \D(\modu \End_R(\mu_N(M)))$ : the (Iyama-Wemyss) mutation functor.
\item[$\bullet$] $\Sym^k_R M$ (resp. $\Sym^k_X \EE$) : $k$-th symmetric product of an $R$-module $M$ (resp. a vector bundle $\EE$ on $X$).

\end{enumerate}

\vspace{3mm}

\noindent
\textbf{Acknowledgements.}
The author would like to express his sincere gratitude to his supervisor Professor Yasunari Nagai for valuable comments and continuous encouragement.
The author is also grateful to Professor Michel Van den Bergh for beneficial conversations and helpful advices, and to Professor Ed Segal for reading the draft of this article and giving many useful comments.
The author would like to thank Hayato Morimura for his careful reading.
Thanks to Morimura's comment on the previous version of this article, the author could improve the presentation in this article and fix many typos.
Finally, the author would like to thank the referee for a lot of helpful suggestions and comments.

This work was done during the author's stay at Hasselt University in Belgium.
This work is supported by Grant-in-Aid for JSPS Research Fellow 17J00857.

\section{Preliminaries} \label{sect: prelim}

\subsection{Abuaf flop}

This section explains the geometry of the Abuaf flop briefly.
For more details, see \cite{Seg16}.
Let $V$ be a four dimensional symplectic vector space.
Let $\LGr(V)$ be the Lagrangian Grassmannian of $V$ and $\Sub \subset V \otimes_{\CC} \stsh_{\LGr(V)}$ the rank two universal subbundle.
Note that $\stsh_{\LGr(V)}(1) := \bigwedge^2 \Sub^*$ is the ample generator of $\Pic(\LGr(V))$,
and this polarisation identifies the Lagrangian Grassmannian $\LGr(V)$ with the quadric threefold $Q_3 \subset \PP^4$.
Note that the canonical embedding $\LGr(V) \subset \Gr(2, V)$ coincides with a hyperplane section $Q_3 = Q_4 \cap H \subset Q_4 \subset \PP^5$.
Let $Y$ be the total space of a vector bundle $\Sub(-1)$:
\[ Y := \Tot(\Sub(-1)) \xrightarrow{\pi} \LGr(V). \]
Since $\bigwedge^2(\Sub(-1)) \simeq \stsh_{\LGr(V)}(-3) \simeq \omega_{\LGr(V)}$, the variety $Y$ is a five dimensional (local) Calabi-Yau variety.

Let $\LGr \subset Y$ be the zero section.
Then, it is possible to contract $\LGr$ and this gives a flopping contraction $\phi : Y \to X$.
The algebra $R := \phi_*\stsh_Y$ is a normal Gorenstein algebra such that $X = \Spec R$.

Next, consider the $3$-dimensional projective space $\PP(V)$.
The symplectic form on $V$ gives an embedding of the universal line bundle $\LL = \stsh_{\PP(V)}(-1)$ into $\Omega_{\PP(V)}^1(1) \simeq \LL^{\bot}$.
Consider  the total space 
\[ Y' := \Tot((\LL^{\bot}/\LL) \otimes \LL^2) \xrightarrow{\pi'} \PP(V) \]
of a vector bundle $(\LL^{\bot}/\LL) \otimes \LL^2$, 
which is another five dimensional (local) Calabi-Yau variety.
The zero section $\PP \subset Y'$ can be contracted, and gives a flopping contraction $\phi' : Y' \to X$.
Combining $Y$ and $Y'$ provides a diagram of a flop
\[ \begin{tikzcd}
Y \arrow[rd, "\phi"] & & Y'. \arrow[ld, "\phi'"'] \\
 & X & 
\end{tikzcd} \]
The affine variety $X$ has a unique singular point  $o \in X$.
In contrast to the case of the Atiyah flop or the Mukai flop, two fibers $\phi^{-1}(o) = \LGr$ and $\phi'^{-1}(o) = \PP$ are not isomorphic to each other.
Since this interesting flop was first provided by Abuaf, we call this flop the \textit{Abuaf flop}.

\subsection{Non-commutative crepant resolution and tilting bundle}

\begin{defi} \label{def NCCR} \rm
Let $R$ be a normal Gorenstein (commutative) algebra, and $M$ a non-zero reflexive $R$-module. 
We set $\Lambda := \End_R(M)$.
We say that the $R$-algebra $\Lambda$ is a \textit{non-commutative crepant resolution (=NCCR)} of $R$, or $M$ gives an NCCR of $R$,
if  $\gldim \Lambda < \infty$ and $\Lambda$ is maximal Cohen-Macaulay as an $R$-module.
\end{defi}

The notion of NCCR is a non-commutative analog of the notion of crepant resolutions.
The following conjecture is due to Bondal, Orlov, and Van den Bergh.

\begin{conj}[\cite{VdB04b}, Conjecture 4.6] \label{BOV}
Let $R$ be a Gorenstein $\mathbb{C}$-algebra.
Then, all crepant resolutions of $R$ and all NCCRs of $R$ are derived equivalent.
\end{conj}

The theory of NCCRs has strong relationship to the theory of tilting bundles.

\begin{defi} \rm
Let $X$ be a variety. A vector bundle $\Tilt$ (of finite rank) on $X$ is called a \textit{partial tilting bundle} if
\begin{enumerate}
\item[(1)] $\Ext_X^i(\Tilt, \Tilt) = 0$ for $i \neq 0$.
\end{enumerate}
Further, if a partial tilting bundle $\Tilt$ satisfies the following condition,
\begin{enumerate}
\item[(2)] $\Tilt$ generates the category $\mathrm{D}(\Qcoh(X))$, i.e. for $E \in \mathrm{D}(\Qcoh(X))$, $\RHom_X(\Tilt, E) = 0$ implies $E = 0$
\end{enumerate}
we say that the bundle $\Tilt$ is a \textit{tilting bundle}.
\end{defi}

\begin{eg} \rm
In \cite{Bei79}, Beilinson showed that the following vector bundles on a projective space $\mathbb{P}^n$
\[ T = \bigoplus_{k=0}^n \stsh_{\mathbb{P}^n}(k), ~~~ T' = \bigoplus_{k=0}^n \Omega_{\mathbb{P}^n}^k(k+1) \]
are tilting bundles.
Note that these tilting bundles come from full strong exceptional collections of the derived category $\D(\PP^n)$ of $\PP^n$ that are called the \textit{Beilinson collections}.
\end{eg}

A tilting bundle on a variety gives an equivalence between the derived category of the variety and the derived category of a non-commutative algebra that is given as the endomorphism ring of the tilting bundle.
This is a generalisation of classical Morita theory.

\begin{thm} \label{equiv tilting}
Let $\Tilt \in \D(X)$ be a tilting bundle on a smooth quasi-projective variety $X$. 
If we set $\Lambda := \End_X(\Tilt)$, we have an equivalence of categories
\[ \RHom_X(\Tilt, -) : \D(X) \xrightarrow{\sim} \D(\modu(\Lambda)), \]
and the quasi-inverse of this functor is given by
\[ - \otimes_{\Lambda} \Tilt : \D(\modu(\Lambda)) \xrightarrow{\sim} \D(X). \]
\end{thm}

For the proof of Theorem \ref{equiv tilting}, see \cite[Theorem 7.6]{HV07} or \cite[Lemma 3.3]{TU10}.
A tilting bundle on a crepant resolution provides an NCCR.

\begin{prop} \label{tilt to NCCR}
Let $\phi : Y \to X = \Spec R$ be a crepant resolution of an affine normal Gorenstein variety $X$.
Let $\Tilt$ be a tilting bundle on $Y$ and assume that $\Tilt$ contains a trivial line bundle $\stsh_Y$ as a direct summand.
Then, we have an isomorphism $\End_Y(\Tilt) \simeq \End_R(\phi_* \Tilt)$.
In particular, the $R$-module $\phi_* \Tilt$ gives an NCCR of $R$.
\end{prop}

To prove this proposition, we need the following three propositions.

%If we discuss about NCCRs, we need to treat reflexive and Cohen-Macaulay modules.
%In the rest of this subsection, we provide some basic properties of reflexive and Cohen-Macaulay modules that we use in later sections.

\begin{prop}[c.f. \cite{H17a} Lemma 3.1]
Let $\phi : Y \to X = \Spec R$ be a crepant resolution of a Gorenstein affine scheme $X$ and $\sh$ be a coherent sheaf on $Y$.
Assume that
\[ H^i(Y, \sh) = 0 = \Ext_Y^i(\sh, \stsh_Y) \]
for all $i > 0$.
Then, an $R$-module $\phi_*\sh$ is Cohen-Macaulay.
\end{prop}

\begin{prop}[see e.g. \cite{H17a} Proposition 2.8] \label{CM ref 1}
Let $R$ be a normal Cohen-Macaulay domain and $M$ a (maximal) Cohen-Macaulay $R$-module.
Then, $M$ is reflexive.
\end{prop}

\begin{prop}[see e.g. \cite{H17a} Proposition 2.9] \label{CM ref 2}
Let $R$ be a normal Cohen-Macaulay domain and $M, N$ (maximal) Cohen-Macaulay $R$-modules.
Then, the $R$-module $\Hom_R(N, M)$ is reflexive.
\end{prop}

\begin{proof}[Proof of Proposition \ref{tilt to NCCR}]
Since $\Tilt$ contains $\stsh_Y$ as a direct summand, we have
\[ H^i(Y, \Tilt) = 0 = \Ext_Y^i(\Tilt, \stsh_Y) \]
for all $i \neq 0$.
Thus $\phi_*\Tilt$ is a Cohen-Macaulay $R$-module and hence $\End_R(\phi_*\Tilt)$ is a reflexive $R$-module.
On the other hand, since $H^i(Y, \Tilt^* \otimes \Tilt) = 0$ for $i \neq 0$ and $(\Tilt^* \otimes \Tilt)^* \simeq \Tilt \otimes \Tilt^*$,
the $R$-module $\phi_*(\Tilt^* \otimes \Tilt) = \End_Y(\Tilt)$ is also Cohen-Macaulay and reflexive.
Since $\End_R(\phi_*\Tilt)$ and $\End_Y(\Tilt)$ are isomorphic to each other in codimension one, we have an isomorphism
\[ \End_R(\phi_*\Tilt) \simeq \End_Y(\Tilt). \]
Since there is an equivalence of categories $\D(Y) \simeq \D(\modu \End_R(\phi_*\Tilt))$, the algebra $\End_R(\phi_*\Tilt)$ has finite global dimension.
\end{proof}

\subsection{Toda-Uehara's construction for tilting bundles and perverse hearts}

Van den Bergh showed in \cite{VdB04a, VdB04b} that if $f : Y \to X$ is a projective morphism that has fibers of dimension at most one and satisfies $Rf_*\stsh_Y \simeq \stsh_X$ (e.g. $3$-fold flopping contraction), 
then there is a tilting bundle on $Y$ that is a projective generator of a perverse heart ${}^0\Per(Y/X)$.
As a generalisation of this result, Toda and Uehara provided a method to construct a tilting bundle in higher dimensional cases that satisfy certain assumptions \cite{TU10}.
They also provided a perverse heart ${}^0\Per(Y/A_{n-1})$ that contains the tilting bundle as a projective generator.
In the present subsection, we recall the construction of Toda-Uehara's tilting bundle.

Let $f : Y \to X = \Spec R$ be a projective morphism from a Noetherian scheme $Y$ to an affine scheme $X$ of finite type.
Assume that $Rf_*\stsh_Y \simeq \stsh_X$ and $\dim f^{-1}(x) \leq n$ for all $x \in X$.
In addition, let us assume the following condition holds for $Y$:

\begin{assum} \label{assumption TU1}
There exists an ample and globally generated line bundle $\stsh_Y(1)$ such that 
\[ H^i(Y, \stsh_Y(-j)) = 0 \]
for $i \geq 2$, $0 < j < n$.
\end{assum}

\noindent
\textbf{Step 1.}
Under this setting, partial tilting bundles $\EE_k$ for $0 \leq k \leq n - 1$ are defined inductively as follows.
First, put $\EE_0 := \stsh_Y$.
Assume that $0 < k \leq n - 1$.
Let $r_{k-1}$ be a minimal number of generators of $\Ext_Y^1(\EE_{k-1}, \stsh_Y(-k))$ over $\End_Y(\EE_{k-1})$.
Take a $r_{k-1}$ generators of $\Ext_Y^1(\EE_{k-1}, \stsh_Y(-k))$ and consider an exact sequence corresponding to the generators:
\[ 0 \to \stsh_Y(-k) \to \mathcal{N}_{k-1} \to \EE_{k-1}^{\oplus r_{k-1}} \to 0. \]
Then $\EE_k := \EE_{k-1} \oplus \mathcal{N}_{k-1}$ is a partial tilting bundle \cite[Claim 4.4]{TU10}.
Repeating this construction gives a partial tilting bundle $\EE_{n-1}$ but this is not a generator in general.

\vspace{3mm}

\noindent
\textbf{Step 2.}
Put $A_{n-1} := \End_Y(\EE_{n-1})$ and consider the following functors
\begin{align*}
F &:= \RHom_Y(\EE_{n-1}, -) : \D(Y) \to \D(\modu A_{n-1}), \\
G &:= - \otimes^{\mathrm{L}}_{A_{n-1}} \EE_{n-1} : \D(\modu A_{n-1}) \to \D(Y).
\end{align*}
Note that $G$ is the left adjoint functor of $F$.
Let us consider an object $F(\stsh_Y(-n)) := \RHom_Y(\EE_{n-1}, \stsh_{Y}(-n))$.
Let $P$ be a projective $A_{n-1}$-resolution of $F(\stsh_Y(-n))$ and $\sigma_{\geq 1}(P)$ the sigma stupid truncation of $P$.
Then, there is a canonical morphism $\sigma_{\geq 1}(P) \to P$.
Furthermore, there is a morphism
\[ G(\sigma_{\geq 1}(P)) \to G(P) \simeq G(F(\stsh_Y(-n))) \xrightarrow{\mathrm{adj}} \stsh_Y(-n). \]
Put $\mathcal{N}_{n-1} := \mathrm{Cone}(G(\sigma_{\geq 1}(P)) \to \stsh_Y(-n))$
and $\EE_n := \EE_{n-1} \oplus \mathcal{N}_{n-1}$.
This $\EE_n$ is a generator of $\D(Y)$ but it is not possible to conclude that $\EE_n$ is tilting in general \cite[Lemma 4.6]{TU10}.

\vspace{3mm}

\noindent
\textbf{Step 3.}
Under the following assumption, we can conclude that $\EE_n$ is tilting.

\begin{assum} \label{assumption TU2} \rm
For an object $\mathcal{K} \in \mathrm{D}(Y)$, if we have
\[ \RHom_Y\left( \bigoplus_{i=0}^{n-1} \stsh_Y(-i), \mathcal{K} \right) = 0, \]
the equality
\[ \RHom_Y\left( \bigoplus_{i=0}^{n-1} \stsh_Y(-i), \mathcal{H}^k(\mathcal{K}) \right) = 0 \]
holds for all $k$.
\end{assum}

\begin{thm}[\cite{TU10}]
Under Assumption \ref{assumption TU2}, $\EE_n$ is a tilting bundle on $Y$.
\end{thm}

\begin{rem}[\cite{TU10}, Remark 4.7] \label{Toda-Uehara lem} \rm
We can also conclude that the object $\EE_n$ is a tilting bundle if we assume the vanishing
\[ H^{> 1}(Y, \stsh_Y(-n)) = 0 \]
instead of Assumption \ref{assumption TU2}.
In this case, the bundle $\mathcal{N}_{n-1}$ lies on an exact sequence
\[ 0 \to \stsh_Y(-n) \to \mathcal{N}_{n-1} \to \EE_{n-1}^{\oplus r_{n-1}} \to 0, \]
where $r_{n-1}$ is the minimal number of generators of $\Ext^1_Y(\EE_{n-1}, \stsh_Y(-n))$ over $A_{n-1}$.
\end{rem}

\noindent
\textbf{Perverse heart.}
Put $\EE := \EE_n$ and $A := \End_Y(\EE)$.
Using the tilting bundle $\EE$ gives a derived equivalence
\[ \Psi_{\EE} := \RHom_Y(\EE, -) : \D(Y) \xrightarrow{\sim} \D(\modu A). \]
In \cite{TU10}, Toda and Uehara also studied the perverse heart
\[ {}^0\Per(Y/A_{n-1}) \subset \D(Y) \]
that corresponds to $\modu A$ under the equivalence $\Psi_{\EE}$.
The construction of ${}^0\Per(Y/A_{n-1})$ is as follows.
First, let us consider a subcategory of $\mathrm{D}(Y)$
\[ \mathrm{D}^{\dagger}(Y) := \{ \mathcal{K} \in \mathrm{D}(Y) \mid F(\mathcal{K}) \in \D(\modu A_{n-1}) \} \]
and set
\begin{align*}
\mathcal{C} &:= \{ \mathcal{K} \in \mathrm{D}(Y) \mid F(\mathcal{K}) = 0 \}, \\
\mathcal{C}^{\leq 0} &:= \mathcal{C} \cap \mathrm{D}(Y)^{\leq 0}, \\
\mathcal{C}^{\geq 0} &:= \mathcal{C} \cap \mathrm{D}(Y)^{\geq 0}.
\end{align*}
By definition, there is an inclusion $i : \mathcal{C} \hookrightarrow \mathrm{D}^{\dagger}(Y)$.
The advantage to consider the subcategory $\mathrm{D}^{\dagger}(Y)$ is that it gives the left and right adjoint of $i$:
\begin{align*}
i^* : \mathrm{D}^{\dagger}(Y) \to \mathcal{C}, ~~ i^! : \mathrm{D}^{\dagger}(Y) \to \mathcal{C}.
\end{align*}
Using these functors, the perverse heart ${}^0\Per(Y/A_{n-1})$ is defined as
\[ {}^0\Per(Y/A_{n-1}) := \{ \mathcal{K} \in \mathrm{D}^{\dagger}(Y) \mid F(\mathcal{K}) \in \modu A_{n-1}, i^*\mathcal{K} \in \mathcal{C}^{\leq 0},  i^!\mathcal{K} \in \mathcal{C}^{\geq 0} \}. \]

\begin{thm}[\cite{TU10} Theorem 5.1]
Under the Assumption \ref{assumption TU2},
the abelian category ${}^0\Per(Y/A_{n-1})$ is the heart of a bounded t-structure on $\D(Y)$, 
and $\Psi({}^0\Per(Y/A_{n-1})) = \modu A$.
In particular, $\EE$ is a projective generator of ${}^0\Per(Y/A_{n-1})$.
\end{thm}

\subsection{Iyama-Wemyss mutation}

In the present subsection, we recall some basic definitions and properties about Iyama-Wemyss mutation.
Iyama-Wemyss's mutation provides a basic tool to compare two different NCCRs.

\begin{defi} \rm
Let $R$ be a normal Gorenstein algebra.
A reflexive $R$-module $M$ is say to be a \textit{modifying module} if $\End_{R}(M)$ is a (maximal) Cohen-Macaulay $R$-module.
\end{defi}

\begin{defi} \label{def approx} \rm
Let $A$ be a ring, $M, N$ $A$-modules, and $N_0 \in \add N$.
A morphism $f : N_0 \to M$ is called a \textit{right $(\add N)$-approximation} if the map
\[ \Hom_A(N, N_0) \xrightarrow{f \circ} \Hom_A(N,M) \]
is surjective.
\end{defi}

Let $R$ be a normal Gorenstein algebra and $M$ a modifying $R$-module.
For any $0 \neq N \in \add M$, consider 
\begin{enumerate}
\item[(1)] a right $(\add N)$-approximation of $M$, $a : N_0 \to M$, and
\item[(2)] a right $(\add N^*)$-approximation of $M^*$, $b : N_1^* \to M^*$.
\end{enumerate}
Put $K_0 := \Ker(a)$ and $K_1 := \Ker(b)$.

\begin{defi} \rm
With notations as above, we define the \textit{right mutation} of $M$ at $N$ to be $\mu_N^R(M) := N \oplus K_0$
and the \textit{left mutation} of $M$ at $N$ to be $\mu_N^L(M) := N \oplus K_1^*$.
\end{defi}

Note that, the right and left mutations are well-defined up to additive closure \cite[Lemma 6.3]{IW14}.
In \cite{IW14}, Iyama and Wemyss proved the following theorem.
 
\begin{thm}[\cite{IW14}] \label{IW mutation equiv}
Let $R$ be a normal Gorenstein algebra
and $M$ a modifying module.
Assume that $0 \neq N \in \add M$.
Then
\begin{enumerate}
\item[(1)] $R$-algebras $\End_R(M)$, $\End_R(\mu_N^R(M))$, and $\End_R(\mu^L_N(M))$ are derved equivalent.
\item[(2)] If $M$ gives an NCCR of $R$, so do its mutations $\mu_N^R(M)$ and $\mu^L_N(M)$.
\end{enumerate}
\end{thm}

The equivalence between $\End_R(M)$ and $\End_R(\mu^L_N(M))$ is given as follows.
Let $Q := \Hom_R(M, N)$ and 
\[ C := \Ima\left(\Hom_R(M, N_1) \to \Hom_R(M, K_1^*)\right). \]
Then, one can show that $V \oplus Q$ is a tilting $\Lambda := \End_R(M)$-module and there is an isomorphism of $R$-algebras
\[ \End_R(\mu_N^L(M)) \simeq \End_{\Lambda}(C \oplus Q). \]
Thus, there is an equivalence
\[ \IW_N := \RHom(C \oplus Q, -) : \D(\modu(\End_R(M))) \to \D(\modu(\End_R(\mu_N^L(M)))). \]
In this paper, we only use left IW mutations and hence we call them simply \textit{IW mutations} and write $\mu_N(M)$ instead of $\mu^L_N(M)$.
We also call the functor $\IW_N$ an \textit{IW mutation functor}.

The following lemmas are useful to find an approximation.

\begin{lem}[\cite{IW14}, Lemma 6.4, (3)] \label{lem approx}
Let us consider a right exact sequence
\[ 0 \to K \xrightarrow{b} N_0 \xrightarrow{a} M, \]
where $a$ is a right $(\add N)$-approximation of $M$.
Then, the dual of the above sequence
\[ 0 \to M^* \xrightarrow{a^*} N_0^* \xrightarrow{b^*} K^* \]
is also right exact and $b^*$ is a right $(\add N^*)$-approximation of $K^*$.
\end{lem}

\begin{lem} \label{lem approx2}
Let $\phi : Y \to X = \Spec R$ be a crepant resolution of an affine Gorenstein normal variety $X$.
Let $\mathcal{W}$ be a vector bundle on $Y$ and
\[ 0 \to \mathcal{K} \to \EE \to \mathcal{C} \to 0 \]
an exact sequence of vector bundles on $Y$.
Assume that 
\begin{enumerate}
\item[(a)] $\EE \in \add(\mathcal{W})$,
\item[(b)] $\mathcal{W} \oplus \mathcal{K}$ and $\mathcal{W} \oplus \mathcal{C}$ are tilting bundles, and
\item[(c)] $\mathcal{W}$ contains $\stsh_Y$ as a direct summand.
\end{enumerate}
Then, 
\begin{enumerate}
\item[(1)] The sequence
\[ 0 \to f_*\mathcal{K} \to f_*\EE \to f_*\mathcal{C} \to 0 \]
is exact and provides a right $(\add f_*\mathcal{W})$-approximation of $f_*\mathcal{C}$.
\item[(2)] The IW mutation functor
\[ \Phi_{f_*\mathcal{W}} : \D(\modu \End_Y(\mathcal{W} \oplus \mathcal{K})) \xrightarrow{\sim} \D( \modu \End_Y(\mathcal{W} \oplus \mathcal{C})) \]
coincides with the functor $\RHom(\RHom_Y(\mathcal{W} \oplus \mathcal{K}, \mathcal{W} \oplus \mathcal{C}), -)$.
\end{enumerate}
\end{lem}

\begin{proof}
First, note that there are isomorphisms of $R$-algebras
\begin{align*}
\End_Y(\mathcal{W} \oplus \mathcal{K}) \simeq \End_R(f_*\mathcal{W} \oplus f_*\mathcal{K}),~
\End_Y(\mathcal{W} \oplus \mathcal{C}) \simeq \End_R(f_*\mathcal{W} \oplus f_*\mathcal{C})
\end{align*}
by Proposition \ref{tilt to NCCR}.
The assumptions (b) and (c) imply $H^1(Y, \mathcal{K}) = 0$, and thus the sequence
\[ 0 \to f_*\mathcal{K} \to f_*\EE \to f_*\mathcal{C} \to 0 \]
is exact.
Moreover, as in the proof of Proposition \ref{tilt to NCCR}, we have
\begin{align*}
\Hom_Y(\mathcal{W}, \EE) \simeq \Hom_R(f_*\mathcal{W}, f_*\EE), ~
\Hom_Y(\mathcal{W}, \mathcal{C}) \simeq \Hom_R(f_*\mathcal{W}, f_*\mathcal{C}).
\end{align*}
Since $\Ext_Y^1(\mathcal{W}, \mathcal{K}) = 0$, the map
\[ \Hom_R(f_*\mathcal{W}, f_*\EE) \to \Hom_R(f_*\mathcal{W}, f_*\mathcal{C}) \]
is surjective.
This shows (1).

Let $V := \Hom_R(f_*\mathcal{W} \oplus f_*\mathcal{K}, f_*\mathcal{W})$ and 
\[ Q := \Ima(\Hom_R(f_*\mathcal{W} \oplus f_*\mathcal{K}, f_*\EE) \to \Hom_R(f_*\mathcal{W} \oplus f_*\mathcal{K}, f_*\mathcal{C})). \]
Then, the IW mutation functor is defined as
\[ \Phi_{f_*\mathcal{W}} := \RHom(V \oplus Q, - ). \]
First, as in the proof of Proposition \ref{tilt to NCCR}, we have
\[ V = \Hom_R(f_*\mathcal{W} \oplus f_*\mathcal{K}, f_*\mathcal{W}) \simeq \Hom_Y(\mathcal{W} \oplus \mathcal{K}, \mathcal{W}) \]
and
\[ \Hom_R(f_*\mathcal{W} \oplus f_*\mathcal{K}, f_*\EE) \simeq \Hom_Y(\mathcal{W} \oplus \mathcal{K}, \EE). \]
Since the $R$-module $\Hom_Y(\mathcal{W} \oplus \mathcal{K}, \mathcal{C})$ is torsion free and isomorphic to $\Hom_R(f_*\mathcal{W} \oplus f_*\mathcal{K}, f_*\mathcal{C})$ in codimension one, the natural map
\[ \Hom_Y(\mathcal{W} \oplus \mathcal{K}, \mathcal{C}) \to \Hom_R(f_*\mathcal{W} \oplus f_*\mathcal{K}, f_*\mathcal{C}) \]
is injective.
Thus, we have the following diagram
\[ \begin{tikzcd}
\Hom_Y(\mathcal{W} \oplus \mathcal{K}, \EE) \arrow[r, equal] \arrow[d, twoheadrightarrow] & \Hom_R(f_*\mathcal{W} \oplus f_*\mathcal{K}, f_*\EE) \arrow[d] \\
\Hom_Y(\mathcal{W} \oplus \mathcal{K}, \mathcal{C}) \arrow[r, hook] & \Hom_R(f_*\mathcal{W} \oplus f_*\mathcal{K}, f_*\mathcal{C})).
\end{tikzcd} \]
Therefore, $Q = \Hom_Y(\mathcal{W} \oplus \mathcal{K}, \mathcal{C})$
and hence $V \oplus Q \simeq \RHom_Y(\mathcal{W} \oplus \mathcal{K}, \mathcal{W} \oplus \mathcal{C})$.
This shows (2).
\end{proof}

\subsection{Spherical twist} \label{susect: def sph}
In this subsection, we recall the definition of spherical twists.

\begin{defi} \rm
Let $X$ be an $n$-dimensional smooth variety.
\begin{enumerate}
\item[(1)] We say that an object $\EE \in \D(X)$ is a \textit{spherical object} if $\EE \otimes \omega_X \simeq \EE$ and
\[ \RHom_X(\EE, \EE) \simeq \CC \oplus \CC[-n]. \]
\item[(2)] Let $\EE$ be a spherical object.
Then a \textit{spherical twist} $\ST_{\EE}$ around $\EE$ is defined as
\[ \ST_{\EE}(\sh) := \Cone(\RHom_X(\EE, \sh) \otimes_{\CC} \EE \to \sh). \]
\end{enumerate}
\end{defi}

%For examples of spherical objects, see Lemma \ref{sph obj}.
It is well-known that a spherical twist gives an autoequivalence of $\D(X)$ (see \cite{ST01}).

%%%%%%%%%%%%%%%%%%%%%%%%%%%%%%%%%%%%%%%%%%%%%%%%%%%%%%%%%%%%%%%%%
\section{Toda-Uehara's tilting bundles and Segal's tilting bundles} \label{sect: mutation}

\subsection{Notations}
From now on, we fix the following notations.
\begin{enumerate}
\item[$\bullet$] $Y := \Tot(\Sub(-1)) \xrightarrow{\pi} \LGr(V)$.
\item[$\bullet$] $Y' := \Tot((\LL^{\bot}/\LL) \otimes \LL^2) \xrightarrow{\pi'} \PP(V)$.
\item[$\bullet$] $\iota : \LGr \hookrightarrow Y$, $\iota' : \PP \hookrightarrow Y'$: the zero sections.
\item[$\bullet$] $Y^o := Y \setminus \LGr \simeq Y' \setminus \PP \simeq X_{\mathrm{sm}}$.
\item[$\bullet$] $\phi : Y \to X$, $\phi' : Y' \to X$ : two crepant resolutions.
We will regard this $Y^o$ as the common open subset of $Y$, $Y'$ and $X$ in the isomorphisms above.
\item[$\bullet$] $\stsh_Y(1) := \pi^* \stsh_{\LGr(V)}(1)$, $\stsh_{Y'}(1) := \pi'^*\stsh_{\PP(V)}(1)$.
\item[$\bullet$] We write $\Sub$ instead of $\pi^*\Sub$.
\end{enumerate}

\subsection{Toda-Uehara's assumptions for $Y$ and $Y'$} \label{subsec: TU assump}

In the present subsection, we check that Toda-Uehara's assumptions (Assumption \ref{assumption TU1} and Assumption \ref{assumption TU2}) hold for $Y$ and $Y'$.
The first assumption follows from Segal's computation.

\begin{lem}[\cite{Seg16}] \label{lem seg}
We have
\begin{enumerate}
\item[(1)] $H^{\geq 1}(Y, \stsh_Y(j)) = 0$ for $j \geq -2$.
\item[(2)] $H^{> 1}(Y', \stsh_{Y'}(j)) = 0$ for $j \geq -3$.
Further, we have $H^1(Y', \stsh_{Y'}(j)) = 0$ for $j \geq -2$ and $H^1(Y', \stsh_{Y'}(-3)) \simeq \CC$.
\end{enumerate}
In particular, pairs $(Y, \stsh_Y(1))$ and $(Y', \stsh_{Y'}(1))$ satisfy Assumption \ref{assumption TU1}.
\end{lem}

\begin{lem}
$Y$ and $Y'$ satisfy Assumption \ref{assumption TU2}.
\end{lem}

The proof of this lemma is almost same as in the one provided in \cite[Section 6.2]{TU10}.

\begin{proof}
First, we provide a proof for $Y$.
By using the bundle $\stsh_Y(1)$, we can embed $Y$ into $\PP^4_R$:
\[ h : Y \to \PP^4_R. \]
Let $g : \PP_R^4 \to X = \Spec R$ be a projection.
Note that the derived category $\D(\PP^4_R)$ has a semi-orthogonal decomposition
\[ \D(\PP^4_R) = \langle g^*\D(X) \otimes \stsh_{\PP^4}(-4), g^*\D(X) \otimes \stsh_{\PP^4}(-3), \cdots, g^*\D(X) \otimes \stsh_{\PP^4} \rangle. \]
Let $\mathcal{K} \in \mathrm{D}(Y)$ and assume that
\[ \RHom_Y \left( \bigoplus_{i=0}^{2} \stsh_Y(-i), \mathcal{K} \right) = 0. \]
Then $h_*\mathcal{K} \in \langle g^*\D(X) \otimes \stsh_{\PP^4}(-4), g^*\D(X) \otimes \stsh_{\PP^4}(-3) \rangle$
and hence there is an exact triangle
\[ g^*W_{-3} \otimes \stsh_Y(-3) \to h_* \mathcal{K} \to g^*W_{-4} \otimes \stsh_Y(-4), \]
where $W_l \in \D(X)$.
Note that the support of $\mathcal{H}^k(h_* \mathcal{K})$ is contained in $Y$ and the support of $\mathcal{H}^k(W_{-4}) \otimes_R \stsh_{\PP^4_R}(-4)$ is the inverse image of a closed subset of $X$ by $g$.
Thus, the map
\[ \mathcal{H}^k(h_* \mathcal{K}) \to \mathcal{H}^k(W_{-4}) \otimes_R \stsh_{\PP^4_R}(-4) \]
should be zero and we have an exact sequence
\[ 0 \to \mathcal{H}^{k-1}(W_{-4}) \otimes_R \stsh_{\PP^4_R}(-4) \to \mathcal{H}^k(W_{-3}) \otimes_R \stsh_{\PP^4_R}(-3) \to
\mathcal{H}^k(h_* \mathcal{K}) \to 0. \]
Using this sequence implies
\[ \RHom_Y\left( \bigoplus_{i=0}^{2} \stsh_Y(-i), \mathcal{H}^k(\mathcal{K}) \right) = 0. \]

Next, we prove the lemma for $Y'$.
Let $\mathcal{K'} \in \mathrm{D}(Y')$ and assume that
\[ \RHom_{Y'} \left( \bigoplus_{i=0}^{2} \stsh_{Y'}(-i), \mathcal{K} \right) = 0. \]
In this case, using an embedding $h' : Y' \hookrightarrow \PP^3_R$ gives
\[ h'_*\mathcal{K}' \in \langle \D(R) \otimes_R \stsh_{\PP^3_R}(-3) \rangle. \]
Thus $\mathcal{H}^k(h'_*\mathcal{K}') \in \langle \D(R) \otimes_R \stsh_{\PP^3_R}(-3) \rangle$, 
which proves the result.
\end{proof}

\begin{cor}
$Y$ (resp. $Y'$) admits a tilting bundle that is a projective generator of the perverse heart ${}^0\Per(Y/A_2)$ (resp. ${}^0\Per(Y'/A'_2)$).
\end{cor}

Explicit descriptions of these tilting bundles are given in the next section.

\subsection{Tilting bundles on $Y$ and $Y'$}

\subsubsection{Tilting bundles on $Y$.}

\begin{thm} \label{tilt on Y}
For $-2 \leq k \leq 1$, let $\Tilt_k$ be a vector bundle
\[ \Tilt_k := \stsh_Y \oplus \stsh_Y(-1) \oplus \stsh_Y(-2) \oplus \Sub(k). \]
Then, $\Tilt_k$ is a tilting bundle on $Y$.
\end{thm}

\begin{proof}
By Lemma \ref{lem seg} (1), the direct sum of line bundles $\stsh_{Y} \oplus \stsh_{Y}(-1) \oplus \stsh_{Y}(-2)$ is a partial tilting bundle on $Y$.
It also follows from \cite{Seg16} that  $\Sub$ is a partial tilting bundle.
Since $\Sub^* \simeq \Sub(1)$, it is enough to show that
\[ H^i(Y', \Sub(j)) = 0 \]
for all $j \geq -2$ and $i >0$.
Adjunction of functors yields
\begin{align*}
H^i(Y', \Sub(j)) &\simeq H^i(\LGr(V), \bigoplus_{l \geq 0} \Sym^l (\Sub^*(1)) \otimes \Sub(j)) \\
&\simeq \bigoplus_{l \geq 0} H^i\left(\LGr(V),  \Sym^l (\Sub) \otimes \Sub \otimes \stsh(2l + j)\right) \\
&\simeq \bigoplus_{l \geq 0} H^i\left(\LGr(V),  \Sym^{l+1} (\Sub)(2l + j) \oplus \Sym^{l-1} (\Sub)(2l + j-1)\right).
\end{align*}
Using the Borel-Bott-Weil theorem as in \cite{Seg16} implies the vanishing of this cohomology for all $i >0$.
%On a Grassmannian $\Gr(2,4)$, the highest weights of bundles
%\[ \Sym^{l+1} (\Sub)(2l + j) ~~ \text{and} ~~ \Sym^{l+1} (\Sub)(2l + j -1) \]
%are $(2l + j, l + j - 1, 0, 0)$ and $(2l + j - 1, l + j - 2, 0 , 0)$.
\end{proof}

\begin{prop} \label{explicit TU bundle}
Let us consider
\[ \Tilt_{\mathrm{T}} := \Tilt_{-2} =  \stsh_{Y} \oplus \stsh_{Y}(-1) \oplus \stsh_{Y}(-2) \oplus \Sub(-2). \]
Then, $\Tilt_{\mathrm{T}}$ coincides with the bundle on $Y$ constructed by Toda-Uehara's method (up to additive closure),
and hence is a projective generator of the perverse heart ${}^0\Per(Y/A_2)$.
\end{prop}

In the proof of this proposition, we use an exact sequences of vector bundles, 
whose existence is proved in Section \ref{sect exc colle} using the theory of exceptional collections.

\begin{proof}
Let $\EE_k$ $(0 \leq k \leq 2)$ be the partial tilting constructed in Toda-Uehara's inductive steps.
Lemma \ref{lem seg} implies $\EE_k = \bigoplus_{i=0}^k \stsh_Y(-i)$.
Put $A_2 := \End_Y(\EE_2)$ and
\[ F := \RHom_Y(\EE_2, -) : \D(Y) \to \D(\modu A_2). \]
The semi-orthogonal decomposition
\[ \D(\LGr(V)) = \langle \Sub(-2), \stsh_{\LGr}(-2), \stsh_{\LGr}(-1), \stsh_{\LGr} \rangle, \]
implies that there is an exact triangle in $\D(\LGr(V))$
\[ \mathcal{G} \to \stsh_{\LGr}(-3) \to \Sub(-2)^{\oplus 4} \to \mathcal{G}[1],\]
where $\mathcal{G} \in \langle \stsh_{\LGr}(-2), \stsh_{\LGr}(-1), \stsh_{\LGr} \rangle$.
Moreover, using Lemma \ref{app lem exact} (1) yields a quasi-isomorphism
\[ \mathcal{G}[1] \simeq_{\mathrm{qis}} (\cdots \to 0 \to \stsh_{\LGr}(-2)^{\oplus 11} \to \stsh_{\LGr}(-1)^{\oplus 5} \to \stsh_{\LGr} \to 0 \to \cdots) \]
of complexes (note that the degree zero term is $\stsh_{\LGr}(-2)^{\oplus 11}$, see Lemma \ref{app lem exact} for the proof).
Pulling back the above triangle to $Y$ by $\pi$ gives an exact triangle
\[ \pi^*\mathcal{G} \to \stsh_{Y}(-3) \to \Sub(-2)^{\oplus 4} \to \pi^*\mathcal{G}[1].\]
Consider the left adjoint functor $G = - \otimes_{A_2}^{\mathbf{L}} \mathcal{E}_2 : \mathrm{D}^-(A_2) \to \mathrm{D}^-(Y)$ of $F$.
Since every term of a complex
\[ \cdots \to 0 \to F(\stsh_{Y}(-2))^{\oplus 11} \to F(\stsh_{Y}(-1))^{\oplus 5} \to F(\stsh_{Y}) \to 0 \to \cdots  \]
is a projective $A_2$-module, the tensor product $G = - \otimes_{A_2}^{\mathbf{L}} \mathcal{E}_2$ applied to this complex does not derive, and hence
\[ G\left(\cdots \to 0 \to F(\stsh_{Y}(-2))^{\oplus 11} \to F(\stsh_{Y}(-1))^{\oplus 5} \to F(\stsh_{Y}) \to 0 \to \cdots \right) \simeq  \pi^*\mathcal{G}[1]. \]
Since $F \circ G \simeq \id$, the complex
\[ \left(\cdots \to 0 \to F(\stsh_{Y}(-2))^{\oplus 11} \to F(\stsh_{Y}(-1))^{\oplus 5} \to F(\stsh_{Y}) \to 0 \to \cdots \right) \]
is a projective resolution of the complex $F(\pi^*\mathcal{G}[1])$.
On the other hand, since $F(\mathcal{S}(-2))^{\oplus 4} \in \modu A_2$, 
there is a projective resolution $(P')^{\bullet}$ of $F(\mathcal{S}(-2))^{\oplus 4}$ such that
$(P')^i = 0$ for $i \geq 1$.

From now on, we construct a projective resolution of $F(\stsh_Y(-3))$ explicitly.
First, there is a chain morphism
\[ \mathcal{S}(-2)^{\oplus 4} \to \left(\cdots \to 0 \to \stsh_{Y}(-2)^{\oplus 11} \to \stsh_{Y}(-1)^{\oplus 5} \to \stsh_{Y} \to 0 \to \cdots \right) \]
whose cone is $\stsh_Y(-3)[1]$.
Applying a functor $F$ gives a morphism
\[ F(\mathcal{S}(-2))^{\oplus 4} \to \left(\cdots \to 0 \to F(\stsh_{Y}(-2))^{\oplus 11} \to F(\stsh_{Y}(-1))^{\oplus 5} \to F(\stsh_{Y}) \to 0 \to \cdots \right) \]
that remain a morphism of chain complexes,
hence there is a morphism of chain complexes
\[ (P')^{\bullet} \to \left(\cdots \to 0 \to F(\stsh_{Y}(-2))^{\oplus 11} \to F(\stsh_{Y}(-1))^{\oplus 5} \to F(\stsh_{Y}) \to 0 \to \cdots \right) \]
whose cone is quasi-isomorphic to $F(\stsh_Y(-3))[1]$.
Thus $F(\stsh_Y(-3))$ is quasi-isomorphic to a complex $P^{\bullet}$ such that
\[ P^i = \begin{cases}
(P')^{i} & \text{if $i \leq 0$} \\
F(\stsh_Y(-2))^{\oplus 11} & \text{if $i =1$} \\
F(\stsh_Y(-1))^{\oplus 5} & \text{if $i=2$} \\
F(\stsh_Y) & \text{if $i=3$} \\
0 & \text{otherwise}.
\end{cases} \]
Since all terms of $P^{\bullet}$ are projective $A_2$-modules, $P^{\bullet}$ is a projective resolution of $F(\stsh_Y(-3))$.
In particular $\sigma_{\geq 1} P^{\bullet} \simeq F(\pi^*\mathcal{G})$,
and hence 
\begin{align*}
G(\sigma_{\geq 1} P^{\bullet}) &\simeq GF(\pi^*\mathcal{G}) \\
&\simeq GFG\left(\cdots \to 0 \to F(\stsh_{Y}(-2))^{\oplus 11} \to F(\stsh_{Y}(-1))^{\oplus 5} \to F(\stsh_{Y}) \to 0 \to \cdots \right)[-1] \\
&\simeq G\left(\cdots \to 0 \to F(\stsh_{Y}(-2))^{\oplus 11} \to F(\stsh_{Y}(-1))^{\oplus 5} \to F(\stsh_{Y}) \to 0 \to \cdots \right)[-1] \\
&\simeq (\pi^* \mathcal{G}[1])[-1] \simeq \pi^*\mathcal{G}.
\end{align*}
Thus the resulting bundle obtained by Toda-Uehara's construction is $\stsh_{Y} \oplus \stsh_{Y}(-1) \oplus \stsh_{Y}(-2) \oplus \Sub(-2)^{\oplus 4}$.
\end{proof}

\begin{defi} \rm
We call the bundle
\[ \Tilt_{\mathrm{T}} := \Tilt_{-2} =  \stsh_{Y} \oplus \stsh_{Y}(-1) \oplus \stsh_{Y}(-2) \oplus \Sub(-2) \]
\textit{Toda-Uehara's tilting bundle} on $Y$.
On the other hand, let us consider a bundle
\[ \Tilt_{\mathrm{S}} := (\Tilt_0)^* \simeq \stsh_Y \oplus \stsh_Y(1) \oplus \stsh_Y(2) \oplus \Sub(1). \]
This tilting bundle coincides with the one found by Segal \cite{Seg16}.
Thus we call this bundle \textit{Segal's tilting bundle} on $Y$.
\end{defi}

\subsubsection{Tilting bundles on $Y'$.}
Since $H^1(Y', \stsh_{Y'}(-3)) \simeq \CC$ by  Lemma \ref{lem seg} (2),
there exists a rank 2 vector bundle $\Sigma$ on $Y'$ that lies in the extension
\[ 0 \to \stsh_{Y'}(-1) \to \Sigma \to \stsh_{Y'}(2) \to 0 \]
corresponding to the generator of
\[ \Ext_{Y'}^1(\stsh_{Y'}(2), \stsh_{Y'}(-1)) \simeq H^1(Y', \stsh_{Y'}(-3)) \simeq \CC. \]

Segal's tilting bundle on $Y'$ is given as follows.

\begin{prop}[\cite{Seg16}]
Put
\[ \Tilt'_{\mathrm{S}} := \stsh_{Y'} \oplus \stsh_{Y'}(-1) \oplus \stsh_{Y'}(-2) \oplus \Sigma(-1), \]
Then, $\Tilt'_{\mathrm{S}}$ is a tilting bundle on $\D(Y')$.
\end{prop}

On the other hand, using Toda-Uehara's construction provides  a new tilting bundle.

\begin{prop}
Put
\[ \Tilt'_{\mathrm{T}} :=  \stsh_{Y'} \oplus \stsh_{Y'}(-1) \oplus \stsh_{Y'}(-2) \oplus \Sigma(-2). \]
Then, $\Tilt'_{\mathrm{T}}$ is Toda-Uehara's tilting bundle on $Y'$,
and hence is a projective generator of the perverse heart ${}^0\Per(Y'/A_2')$,
where $A_2'$ is the endomorphism ring of a vector bundle $\stsh_{Y'} \oplus \stsh_{Y'}(-1) \oplus \stsh_{Y'}(-2)$.
\end{prop}

\begin{proof}
If $k \leq 2$, then $H^i(Y', \stsh_{Y'}(-k)) = 0$ for all $i \geq 1$.
Furthermore, $H^1(Y', \stsh_{Y'}(-3)) \simeq \CC$ and $H^i(Y', \stsh_{Y'}(-3)) = 0$ for $i  \geq 2$.
Recall that the vector bundle $\Sigma(-2)$ lies in an exact sequence 
\[ 0 \to \stsh_{Y'}(-3) \to  \Sigma(-2) \to \stsh_{Y'} \to 0 \]
that corresponds to the generator of $H^1(Y', \stsh_{Y'}(-3))$.
Since
\[ \Ext_{Y'}^1(\stsh_{Y'} \oplus \stsh_{Y'}(-1) \oplus \stsh_{Y'}(-2), \stsh_{Y'}(-3)) \simeq H^1(Y', \stsh_{Y'}(-3)), \]
the bundle $\Tilt_{\mathrm{T}}'$ is the Toda-Uehara's tilting bundle on $Y'$ and a projective generator of the perverse heart ${}^0\Per(Y'/A_2')$
by Remark \ref{Toda-Uehara lem}.
\end{proof}

\subsection{Derived equivalences for the Abuaf flop}
In this section, we define derived equivalences induced by tilting bundles.
Put
\[ \Tilt_{\mathrm{U}} := (\Tilt_{-1})^* \simeq \stsh_Y \oplus \stsh_Y(1) \oplus \stsh_Y(2) \oplus \Sub(2). \]

\begin{lem} \label{identification}
The following isomorphism of tilting bundles on $Y^o$ holds.
\begin{enumerate}
\item[(1)] $\Tilt_{\mathrm{S}}|_{Y^o} \simeq \Tilt'_{\mathrm{S}}|_{Y^o}$.
\item[(2)] $\Tilt_{\mathrm{U}}|_{Y^o} \simeq \Tilt'_{\mathrm{T}}|_{Y^o}$.
\end{enumerate}
Thus, the following isomorphism of $R$-algebras holds.
\begin{enumerate}
\item[(i)] $\End_{Y}(\Tilt_{\mathrm{S}}) \simeq \End_{Y'}(\Tilt'_{\mathrm{S}})$.
\item[(ii)] $\End_{Y}(\Tilt_{\mathrm{U}}) \simeq \End_{Y'}(\Tilt'_{\mathrm{T}})$.
\end{enumerate}
\end{lem}

\begin{proof}
In \cite{Seg16}, Segal proved that
\[ \stsh_Y(a)|_{Y^o} \simeq \stsh_{Y'}(-a)|_{Y^o} ~ \text{ and } ~ \Sub|_{Y^o} \simeq \Sigma|_{Y^o}. \]
The result follows from these isomorphisms.
The second statement follows from the fact that the endomorphism rings are reflexive $R$-modules.
\end{proof}

\begin{rem} \rm
The vector bundle $\Tilt_{\mathrm{T}}|_{Y^o}$ on $Y^o$ extends to an bundle
\[ \stsh_{Y'} \oplus \stsh_{Y'}(1) \oplus \stsh_{Y'}(2) \oplus \Sigma(2) \]
on $Y'$.
Unfortunately, this bundle is not tilting.
\end{rem}

\begin{defi} \rm
We set 
\begin{align*}
\Lambda_{\mathrm{T}} &:= \End_Y(\Tilt_{\mathrm{T}}), \\
\Lambda_{\mathrm{S}} &:= \End_Y(\Tilt_{\mathrm{S}}) = \End_{Y'}(\Tilt'_{\mathrm{S}}), \\
\Lambda_{\mathrm{U}} &:= \End_Y(\Tilt_{\mathrm{U}}) = \End_{Y'}(\Tilt'_{\mathrm{T}}),
\end{align*}
and
\begin{align*}
\Psi_{\mathrm{T}} &:= \RHom_Y(\Tilt_{\mathrm{T}}, -) : \D(Y) \xrightarrow{\sim} \D(\modu \Lambda_{\mathrm{T}}), \\
\Psi_{\mathrm{S}} &:= \RHom_Y(\Tilt_{\mathrm{S}}, -) : \D(Y) \xrightarrow{\sim} \D(\modu \Lambda_{\mathrm{S}}), \\
\Psi_{\mathrm{U}} &:= \RHom_Y(\Tilt_{\mathrm{U}}, -) : \D(Y) \xrightarrow{\sim} \D(\modu \Lambda_{\mathrm{U}}), \\
\Psi'_{\mathrm{T}} &:= \RHom_{Y'}(\Tilt'_{\mathrm{T}}, -) : \D(Y') \xrightarrow{\sim} \D(\modu \Lambda_{\mathrm{U}}), \\
\Psi'_{\mathrm{S}} &:= \RHom_{Y'}(\Tilt'_{\mathrm{S}}, -) : \D(Y') \xrightarrow{\sim} \D(\modu \Lambda_{\mathrm{S}}).
\end{align*}
\end{defi}

\begin{defi} \rm
Consider equivalences of categories that are given as
\begin{align*}
\mathrm{Seg} &:= (\Psi'_{\mathrm{S}})^{-1} \circ \Psi_{\mathrm{S}} : \D(Y) \xrightarrow{\sim} \D(Y'), \\
\mathrm{Seg}' &:= \Seg^{-1} = (\Psi_{\mathrm{S}})^{-1} \circ \Psi'_{\mathrm{S}} : \D(Y') \xrightarrow{\sim} \D(Y).
\end{align*}
These equivalences are introduced by Segal \cite{Seg16}.
Hence we call these functors \textit{Segal's equivalences}.
On the other hand, there are equivalences
\begin{align*}
\TU' &:= \Psi_{\mathrm{U}}^{-1} \circ \Psi'_{\mathrm{T}} := \D(Y') \to \D(Y) \\
\UT &:= \TU'^{-1} = \Psi'^{-1}_{\mathrm{T}} \circ \Psi_{\mathrm{U}} : \D(Y) \to \D(Y').
\end{align*}
Since we construct these equivalence by using Toda-Uehara's tilting bundle on $Y'$,
we call these equivalences $\TU'$ and $\UT$ \textit{Toda-Uehara's equivalences}.
\end{defi}

\subsection{Segal's tilting vs Toda-Uehara's tilting}

In this subsection, we compare Toda-Uehara's tilting bundles with Segal's by using IW mutations.
We will use the following lemma.

\begin{lem} \label{lem tilt1}
Let $\mathcal{W}$ be a vector bundle on a smooth variety $Z$ and 
\[ 0 \to \EE_0 \xrightarrow{a_0} \EE_1 \xrightarrow{a_1} \EE_2 \xrightarrow{a_2} \cdots \xrightarrow{a_{m-2}} \EE_{m-1} \xrightarrow{a_{m-1}} \EE_m \to 0 \]
a long exact sequence consisting of vector bundles $\EE_k$ ($0 \leq k \leq m$) on $Z$.
Assume that 
\begin{enumerate}
\item[(a)] $\mathcal{W} \oplus \EE_0$ and $\mathcal{W} \oplus \EE_m$ are tilting bundles.
\item[(b)] $\EE_k \in \add(\mathcal{W})$ for $1 \leq k \leq m - 1$.
%\item[(c)] $\Ext^i(\Ima(a_k), \Ima(a_{k-1})) = 0$ for $ i \geq 2$ and $1 \leq k \leq m - 1$.
\end{enumerate}
Then, $\mathcal{W} \oplus \Ima(a_k)$ is a tilting bundle for all $0 \leq k \leq m - 1$.
\end{lem}

\begin{proof}
Since $\mathcal{W} \oplus \EE_0$ is a tilting bundle, $\Ext_Z^i(\mathcal{W}, \EE_0) = \Ext_Z^i(\mathcal{W}, \Ima(a_0))= 0$ for $i \geq 1$.
Let $k > 0$ and assume $\Ext^i(\mathcal{W}, \Ima(a_{k-1})) = 0$ for $i \geq 1$.
Then the exact sequence
\[ 0 \to \Ima(a_{k-1}) \to \EE_k \to \Ima(a_k) \to 0, \]
and the assumption (b) imply $\Ext_Z^i(\mathcal{W}, \Ima(a_k)) = 0$ for all $i \geq 1$.
Thus, $\Ext_Z^i(\mathcal{W}, \Ima(a_k)) = 0$ for all $i \geq 1$ and $0 \leq k \leq m - 1$.

Similarly, using the assumption that $\mathcal{W} \oplus \EE_m$ is a tilting bundle implies
$\Ext_Z^i(\Ima(a_k), \mathcal{W}) = 0$ for all $i \geq 1$ and $0 \leq k \leq m - 1$.

Next,  assume that $\Ima(a_{k-1})$ is partial tilting.
Consider the exact sequence
\[ 0 \to \Ima(a_{k-1}) \to \EE_k \to \Ima(a_k) \to 0. \]
Applying the functor $\RHom(-, \Ima(a_{k-1}))$ to this sequence gives an exact triangle
\[ \RHom_Z(\Ima(a_k), \Ima(a_{k-1})) \to \RHom_Z(\EE_k, \Ima(a_{k-1})) \to \RHom_Z(\Ima(a_{k-1}), \Ima(a_{k-1})). \]
The assumption (b) and the above arguments shows $\Ext_Z^i(\EE_k, \Ima(a_{k-1})) = 0$ for all $i \geq 1$.
Therefore, $\Ext_Z^i(\Ima(a_k), \Ima(a_{k-1})) = 0$ for all $ i \geq 2$.

Consider the sequence 
\[ 0 \to \Ima(a_{k-1}) \to \EE_k \to \Ima(a_k) \to 0 \]
again, and then applying the functor $\RHom_Z(\Ima(a_k), -)$ to this sequence gives a triangle
\[ \RHom_Z(\Ima(a_k), \Ima(a_{k-1})) \to \RHom_Z(\Ima(a_k), \EE_k) \to \RHom_Z(\Ima(a_k), \Ima(a_k)). \]
The assumption (b) and the above arguments imply $\Ext^i(\Ima(a_k), \EE_k) = 0$ for all $i \geq 1$.
Thus, from the above computation, we have $\Ext_Z^i(\Ima(a_k), \Ima(a_k)) = 0$ for all $i \geq 1$.

It is clear that $\mathcal{W} \oplus \Ima(a_k)$ is a generator.
Thus, the bundle $\mathcal{W} \oplus \Ima(a_k)$ is tilting.
\end{proof}

For the next theorem, we fix the following notations:
\begin{align*}
M_a := \phi_*\stsh_Y(a), ~S_a := \phi_*\Sub(a).
\end{align*}
Note that $M_0 = R$.
First, we compare two NCCRs $\Lambda_{\mathrm{T}}$ and $\Lambda_{\mathrm{S}}$.

\begin{thm} \label{mutation1}
A derived equivalence of NCCRs
\[ \Psi_{\mathrm{S}} \circ \Psi_{\mathrm{T}}^{-1} \simeq \RHom_{\Lambda_{\mathrm{T}}}(\RHom_Y(\Tilt_{\mathrm{T}}, \Tilt_{\mathrm{S}}), -) : \D(\modu \Lambda_{\mathrm{T}}) \xrightarrow{\sim} \D(\modu \Lambda_{\mathrm{S}}) \]
can be written as a composition of nine IW mutation functors.
\end{thm}

As in the proof of Proposition \ref{explicit TU bundle}, we will use some exact sequences from Lemma \ref{app lem exact}.

\begin{proof}%[Proof of Theorem \ref{mutation1}]
Put
\begin{align*}
\nu \Tilt &:= \stsh_{Y} \oplus \stsh_{Y}(-1) \oplus \stsh_{Y}(-2) \oplus \Sub(1) \\
\nu^2 \Tilt &:= \stsh_{Y} \oplus \stsh_{Y}(-1) \oplus \stsh_{Y}(1) \oplus \Sub(1).
\end{align*}
By Theorem \ref{tilt on Y}, these bundles are tilting.
Set
\begin{align*}
W_1 := R \oplus M_{-1} \oplus M_{-2},~
W_2 := R \oplus M_{-1} \oplus S_1,~
W_3 := R \oplus S_1 \oplus M_1.
\end{align*}
We will show that there are three isomorphisms
\begin{align*}
&\mu_{W_1}\mu_{W_1}\mu_{W_1}(\End_Y(\Tilt_{\mathrm{T}})) \simeq \End_Y(\nu \Tilt), \\
&\mu_{W_2}\mu_{W_2}\mu_{W_2}(\End_Y(\nu \Tilt)) \simeq \End_Y(\nu^2 \Tilt), \\
&\mu_{W_3}\mu_{W_3}\mu_{W_3}(\End_Y(\nu^2 \Tilt)) \simeq \End_Y(\Tilt_{\mathrm{S}}),
\end{align*}
and each IW mutation functors can be written as
\begin{align*}
&\IW_{W_1}^3 \simeq \RHom_{\Lambda_{\mathrm{T}}}(\RHom(\Tilt_{\mathrm{T}}, \nu \Tilt), -) : \D(\Lambda_{\mathrm{T}}) \xrightarrow{\sim} \D(\modu  \End_Y(\nu \Tilt)), \\
&\IW_{W_2}^3 \simeq \RHom_{\End_Y(\nu \Tilt)}(\RHom(\nu \Tilt, \nu^2 \Tilt), -) : \D(\modu \End_Y(\nu \Tilt)) \xrightarrow{\sim} \D(\modu  \End_Y(\nu^2 \Tilt)), \\
&\IW_{W_3}^3 \simeq \RHom_{\End_Y(\nu^2 \Tilt)}(\RHom(\nu^2 \Tilt, \Tilt_{\mathrm{S}}), -) : \D(\modu \End_Y(\nu^2 \Tilt)) \xrightarrow{\sim} \D(\Lambda_{\mathrm{S}}).
\end{align*}
First we provide the proof for mutations at $W_1$.
Recall that there exists an exact sequence $0 \to \Sub(k) \to \stsh_Y^{\oplus 4} \to \Sub(k+1) \to 0$ for all $k \in \mathbb{Z}$.
Composing these exact sequences for $-2 \leq k \leq 0$ gives a long exact sequence
\[ 0 \to \Sub(-2) \xrightarrow{a_{-2}} \stsh_Y(-2)^{\oplus 4} \xrightarrow{a_{-1}} \stsh_Y(-1)^{\oplus 4} \xrightarrow{a_0} \stsh_Y^{\oplus 4} \xrightarrow{a_1} \Sub(1) \to 0. \]
%Note that the image of the map $a_i$ is $\Sub(i)$.
Pushing out this exact sequence to $X$ gives an exact sequence
\[ 0 \to S_{-2} \xrightarrow{a_{-2}} M_{-2}^{\oplus 4} \xrightarrow{a_{-1}} M_{-1}^{\oplus 4} \xrightarrow{a_0} M_0^{\oplus 4} \xrightarrow{a_1} S_1 \to 0, \]
whose splices are
\[ 0 \to S_i  \xrightarrow{a_{i}} M_{i}^{\oplus 4} \xrightarrow{a_{i+1}} S_{i+1} \to 0. \]
for $- 2 \leq i \leq 0$.
By Lemma \ref{lem approx2}, the morphism $a_{i+1}$ is a right $(\add W_1)$-approximation of $S_{i+1}$ for $-2 \leq i \leq 0$ and
\[ \mu_{W_1}(W_1 \oplus S_i) = W_1 \oplus S_{i+1}. \]
Let $Q_i := \Hom_R(W_1 \oplus S_i, W_1)$ and
\[ C_i := \mathrm{Im}(\Hom_R(W_1 \oplus S_i, M_i^{\oplus 4}) \to \Hom_R(W_1 \oplus S_i, S_{i+1})). \]
Then, IW mutation functor
\[ \Phi_{W_1} : \D(\modu \End_R(W_1 \oplus S_i)) \to \D(\modu \End_R(W_1 \oplus S_{i+1})) \]
is given by
\[ \Phi_{W_1}(-) := \RHom_{\End_R(W_1 \oplus S_i)}(Q_i \oplus C_i, -).  \]
Again, by Lemma \ref{lem approx2}, there is an isomorphism
\[ \RHom_Y(\mathcal{W}_1 \oplus \Sub(i), \mathcal{W}_1 \oplus \Sub(i+1)) \simeq Q_i \oplus C_i \]
for $-2 \leq i \leq 0$ and hence the following diagram commutes
\[ \begin{tikzcd}
\D(Y) \arrow[d, "\Psi_i"'] \arrow[rd, "\Psi_{i+1}"] & \\
\D(\modu \End_Y(\mathcal{W}_1 \oplus \Sub(i))) \arrow[r, "\IW_{W_1}"'] & \D(\modu \End_Y(\mathcal{W}_1 \oplus \Sub(i+1))),
\end{tikzcd} \]
where $\Psi_i := \RHom_Y(\mathcal{W}_1 \oplus \Sub(i), - )$.
Therefore 
\[ \IW_{W_1}^3 \simeq \Psi_1 \circ \Psi_{-2}^{-1} \simeq \RHom_{\Lambda_{\mathrm{T}}}(\RHom(\Tilt_{\mathrm{T}}, \nu \Tilt), -). \]

To show the result for $W_2$, we use an exact sequence in Lemma \ref{app lem exact} (2)
\[ 0 \to \stsh_Y(-2) \xrightarrow{b_1} \stsh_Y(-1)^{\oplus 5} \xrightarrow{b_2} \stsh_Y^{\oplus 11} \xrightarrow{b_3} \Sub(1)^{\oplus 4} \xrightarrow{b_4} \stsh_Y(1) \to 0. \]
Put $\mathcal{W}_2 := \stsh_Y(-1) \oplus \stsh_Y \oplus \Sub(1)$.
Then, $\mathcal{W}_2 \oplus \stsh_Y(-2)$ and $\mathcal{W}_2 \oplus \stsh_Y(1)$ are tilting bundles by Theorem \ref{tilt on Y}.
Therefore, by Lemma \ref{lem tilt1}, the bundle $\mathcal{W}_2 \oplus \Cok(b_j)$ is also a tilting bundle for all $1 \leq j \leq 4$.
Then the same argument as in the case of $W_1$ shows the result.

One can show for $W_3$ by using the same argument.
We note that the exact sequence we use in this case is 
\[ 0 \to \stsh_Y(-1) \xrightarrow{c_1} \stsh_Y^{\oplus 5} \xrightarrow{c_2} \Sub(1)^{\oplus 4} \xrightarrow{c_3} \stsh_Y(1)^{\oplus 5} \xrightarrow{c_4} \stsh_Y(2) \to 0, \]
which is provided in Lemma \ref{app lem exact} (3), and the cokernels are given by
\begin{align*}
\Cok(c_1) \simeq \pi^*(T_{\PP^4}(-1)|_{\LGr}),~
\Cok(c_2) \simeq \pi^*(\Omega_{\PP^4}^1(2)|_{\LGr}).
\end{align*}
\end{proof}

Next, we compare $\Lambda_{\mathrm{S}}$ with $\Lambda_{\mathrm{U}}$.
The IW mutation that connects $\Lambda_{\mathrm{S}}$ and $\Lambda_{\mathrm{U}}$ is much simpler than the one that connects $\Lambda_{\mathrm{T}}$ and $\Lambda_{\mathrm{S}}$.

\begin{thm} \label{mutation2}
Let $W_4 := M_0 \oplus M_1 \oplus M_2$.
$\Lambda_{\mathrm{U}}$ is a left IW mutation of $\Lambda_{\mathrm{S}}$ at $W_4$.
Furthermore, if we set the IW functor
\[ \IW_{W_4} : \D(\modu \Lambda_{\mathrm{S}}) \xrightarrow{\sim} \D(\modu \Lambda_{\mathrm{U}}), \]
then the following diagram commutes
\[ \begin{tikzcd}
\D(Y) \arrow[r, "\Psi_{\mathrm{S}}"] \arrow[rd, "\Psi_{\mathrm{U}}"'] & \D(\modu \Lambda_{\mathrm{S}}) \arrow[d, "\IW_{W_4}"] \\
 & \D(\modu \Lambda_{\mathrm{U}}).
\end{tikzcd} \]
\end{thm}

\begin{proof}
Consider an exact sequence
\[ 0 \to S_1 \to V \otimes_{\CC} M_1 \to S_2 \to 0 \]
obtained by pushing an exact sequence
\[ 0 \to \Sub(1) \to V \otimes_{\CC} \stsh_Y(1) \to \Sub(2) \to 0 \]
on $Y$ by $\phi$.
Then, by Lemma \ref{lem approx2}, this sequence is a right $(\add W_4)$-approximation of $S_2$ and
we have $\mu_{W_4}(W_4 \oplus S_1) = W_4 \oplus S_2$.
The commutativity of the diagram also follows from Lemma \ref{lem approx2}.
\end{proof}

Summarizing the above results, we have the following corollary.

\begin{cor}
Let $\IW$ be an equivalence between $\D(\modu \Lambda_{\mathrm{T}})$ and $\D(\modu \Lambda_{\mathrm{U}})$ obtained by composing ten IW mutation functors:
\[ \IW := \IW_{W_4} \circ \IW_{W_3} \circ \IW_{W_3} \circ \IW_{W_3} \circ \IW_{W_2} \circ \IW_{W_2} \circ \IW_{W_2} \circ \IW_{W_1} \circ \IW_{W_1} \circ \IW_{W_1}. \]
The equivalence between $\D(Y)$ and $\D(Y')$ obtained by a composition
\[ \D(Y) \xrightarrow{\Psi_{\mathrm{T}}} \D(\modu \Lambda_{\mathrm{T}}) \xrightarrow{\IW} \D(\modu \Lambda_{\mathrm{U}}) \xrightarrow{\Psi_{\mathrm{T}}^{-1}} \D(Y') \]
is the inverse of the functor $\TU'$.
\end{cor}

Later, we show that the Fourier-Mukai kernel of the functor $\TU'$ is the structure sheaf of $\widetilde{Y} \cup_E (\LGr \times \PP)$,
where $\widetilde{Y}$ is the blowing up of $Y$ (or $Y'$) along the zero section and $E$ is the exceptional divisor.
 %Please compare this corollary with \cite[Theorem 4.2]{We14} and \cite[Corollary 5.14]{H17a}.

%\[ \begin{tikzcd}
% & \D(Y) \arrow[d, equal] \arrow[r, "\Psi_{\mathrm{T}}"] & \D(\modu \Lambda_{\mathrm{T}}) \arrow[d] &  \\
%\text{Segal's equivalence:} & \D(Y) \arrow[d, equal] \arrow[r, "\Psi_{\mathrm{S}}"] & \D(\modu \Lambda_{\mathrm{S}}) \arrow[r, "(\Psi'_{\mathrm{S}})^{-1}"] \arrow[d] & \D(Y')  \\
%\text{Toda-Uehara's equivalence:} & \D(Y) \arrow[r, "\Psi_{\mathrm{U}}"] & \D(\modu \Lambda_{\mathrm{U}}) \arrow[r, "(\Psi'_{\mathrm{T}})^{-1}"] & \D(Y')
%\end{tikzcd} \]

%%%%%%%%%%%%%%%%%%%%%%%%%%%%%%%%%%%%%%%%%%%%%%%%%%%%%%%%%%%%%%%%%
\section{Flop-Flop=Twist results and Multi-mutation=Twist result} \label{sect: twist}

In this section, we show ``flop-flop=twist" results and ``multi-mutation=twist" results for the Abuaf flop.

\subsection{spherical objects}
First, we study spherical objects on $Y$ and $Y'$.
For the definition of spherical objects and spherical twists, see Section \ref{susect: def sph}.

\begin{lem} \label{sph obj}
\begin{enumerate}
\item[(1)] Let $\iota : \LGr \hookrightarrow Y$ be the zero section.
Then, an object $\iota_* \stsh_{\LGr} \in \D(Y)$ is a spherical object.
\item[(2)] Let $\iota' : \PP \hookrightarrow Y'$ be the zero section.
Then, an object $\iota'_* \stsh_{\PP} \in \D(Y')$ is a spherical object.
\end{enumerate}
\end{lem}

\begin{proof}
Here we provide the proof of (2) only, but one can show (1) by using the same argument.
The normal bundle $\mathcal{N}_{\PP/Y'}$ of the zero section is isomorphic to $(\LL^{\bot}/\LL) \otimes \LL^2$.
Note that this bundle lies on the exact sequence
\[ 0 \to \stsh_{\PP}(-3) \to \Omega_{\PP}^1(-1) \to \mathcal{N}_{\PP/Y'} \to 0. \]
Thus, we have
\begin{align*}
&\RGamma(\PP, \bigwedge^0  \mathcal{N}_{\PP/Y'}) \simeq \mathbb{C}, \\
&\RGamma(\PP, \bigwedge^1  \mathcal{N}_{\PP/Y'}) \simeq 0, ~ ~ \text{and} \\
&\RGamma(\PP, \bigwedge^2  \mathcal{N}_{\PP/Y'}) \simeq \RGamma(\PP, \stsh_{\PP}(-4)) \simeq \mathbb{C}[-3].
\end{align*}
Let us consider a spectral sequence
\[ E_2^{p, q} := H^p(Y', \lext_{Y'}^q(\iota'_* \stsh_{\PP}, j'_*\stsh_{\PP})) \Rightarrow E^{p+q} = \Ext_{Y'}^{p+q}(\iota'_* \stsh_{\PP}, \iota'_*\stsh_{\PP}). \]
Since we have an isomorphism 
\[ \lext_{Y'}^q(\iota'_* \stsh_{\PP}, \iota'_*\stsh_{\PP}) \simeq \iota'_*\bigwedge^q  \mathcal{N}_{\PP/Y'}, \]
we have
\begin{align*}
E_2^{p,q} = \begin{cases}
\mathbb{C} & \text{if $p=q=0$ or $p=3, q= 2$}, \\
0 & \text{otherwise}.
\end{cases}
\end{align*}
Therefore, we have
\begin{align*}
\Ext_{Y'}^{i}(\iota'_* \stsh_{\PP}, \iota'_*\stsh_{\PP}) = \begin{cases}
\mathbb{C} & \text{if $i=0$ or $i=5$}, \\
0 & \text{otherwise}.
\end{cases}
\end{align*}
Since $Y'$ is Calabi-Yau, the condition $\iota'_*\stsh_{\PP} \otimes \omega_{Y'} \simeq \iota'_*\stsh_{\PP}$ is trivially satisfied.
Hence the object $\iota'_*\stsh_{\PP}$ is a spherical object.
\end{proof}

\subsection{On the side of $Y'$}
In the present subsection, we prove a ``flop-flop=twist" result on the side of $Y'$.
The next lemma is a key of the proof of Theorem \ref{thm flop-flop}, which provides a ``flop-flop=twist" result.
Let $\widetilde{Y}$ be a blowing up of $Y$ along the zero section $\LGr$. 
This is isomorphic to the blowing up of $Y'$ along the zero section $\PP$, and these two blowups give the same exceptional divisor $E \subset \widetilde{Y}$ \cite{Seg16}.
 
\begin{lem} \label{key lem ST}
There is an exact sequence
\[ 0 \to \Sigma(-1) \to V \otimes_{\CC} \stsh_{Y'}(-1) \to \Sigma(-2) \to \iota'_*\stsh_{\PP}(-3) \to 0 \]
on $Y'$.
\end{lem}

\begin{proof}
On $Y$, there is a canonical exact sequence
\[ 0 \to \Sub(1) \to V \otimes_{\CC} \stsh_Y(1) \to \Sub(2) \to 0. \]
By restricting this sequence on $Y^o$ and then extending it to $Y'$, we have a left exact sequence
\[ 0 \to \Sigma(-1) \to V \otimes_{\CC} \stsh_{Y'}(1) \xrightarrow{a} \Sigma(2). \]
Thus, it is enough to show that $\Cok(a) \simeq \stsh_{\PP}(-3)$.

Consider two open immersions $j' : Y^o \hookrightarrow Y'$ and $\tilde{j} : Y^o \hookrightarrow \widetilde{Y}$.
Since $\tilde{j}$ is an affine morphism, there are an exact sequence
\[ 0 \to \tilde{j}_* \Sigma(-1)|_{Y^o} \to V \otimes_{\CC} \tilde{j}_*\stsh_{Y'}(1)|_{Y^o} \to \tilde{j}_*\Sigma(2)|_{Y^o} \to 0 \]
and an isomorphism
\[ \RR j'_*(\Sigma(-1)|_{Y^o}) \simeq \RR\bar{p}_* \tilde{j}_*\Sigma(-1)|_{Y^o}. \]
On the other hand, there exists an exact sequence
\[ 0 \to \stsh_{\widetilde{Y}} \to \tilde{j}_* \stsh_{Y^o} \to \bigoplus_{d=1}^{\infty} \stsh_E(dE) \to 0. \]
This exact sequence together with the projection formula implies
\begin{align*}
\RR^1 j'_*(\Sigma(-1)|_{Y^o}) &\simeq \RR^1 \bar{p}_* \tilde{j}_*\Sigma(-1)|_{Y^o} \\
&\simeq \Sigma(-1)|_{\PP} \otimes \bigoplus_{d \geq 1} \RR^1\bar{p}_* \stsh_E(dE) \\
&\simeq \Sigma(-1)|_{\PP} \otimes \bigoplus_{d \geq 1} \Sym^{d-2}(\LL^{\bot}/\LL) \otimes \LL^{2d} \\
&\simeq \bigoplus_{d \geq 1} \left(\Sym^{d-2}(\LL^{\bot}/\LL) \otimes \LL^{2d+2} \right) \oplus \left(\Sym^{d-2}(\LL^{\bot}/\LL) \otimes \LL^{2d-1} \right).
\end{align*}
Since the sheaf $\Cok(a)$ is a subsheaf of $\RR^1 j'_*(\Sigma(-1)|_{Y^o})$, the map $\Sigma(-2) \to \RR^1 j'_*(\Sigma(-1)|_{Y^o})$ factors through as
\[ \Sigma(-2) \to \Sigma(-2)|_{\PP} \twoheadrightarrow \Cok(a) \hookrightarrow \RR^1 j'_*(\Sigma(-1)|_{Y^o}). \]
Note that $\Sigma(-2)|_{\PP} = \stsh_{\PP} \oplus \stsh_{\PP}(-3)$.
It is easy to observe that two sheaves (on $\PP$)
\[ \Sym^{d-2}(\LL^{\bot}/\LL) \otimes \LL^{2d+2} ~~ \text{ and } ~~ \Sym^{d-2}(\LL^{\bot}/\LL) \otimes \LL^{2d-1} \]
do not have global sections for all $d \geq 1$.
Thus, $\Cok(a)$ is a torsion free sheaf on $\PP$ that can be written as a quotient of $\stsh_{\PP}(-3)$,
which means $\Cok(a) \simeq \stsh_{\PP}(-3)$.
\end{proof}

\begin{thm} \label{thm flop-flop}
We have a functor isomorphism
\[ \UT \circ \Seg' \simeq \ST_{\iota'_* \stsh_{\PP}(-3)} \in \Auteq(\D(Y')). \]
\end{thm}

\begin{proof}
We have to show the following diagram commutes
\[ \begin{tikzcd}
\D(Y') \arrow[r, "\ST_{\iota'_* \stsh_{\PP}(-3)}^{-1}"] \arrow[d, "\TU'"'] & \D(Y') \arrow[d, "\Psi'_{\mathrm{S}}"] \\
\D(Y) \arrow[r, "\Psi_{\mathrm{S}}"] &  \D(\modu \Lambda_{\mathrm{S}}). 
\end{tikzcd} \]
Note that there are isomorphisms of equivalence functors
\begin{align*}
\Psi'_{\mathrm{S}} \circ \ST_{\iota'_* \stsh_{\PP}(-3)}^{-1} &\simeq \RHom_{Y'}(\ST_{\iota'_* \stsh_{\PP}(-3)}(\Tilt'_{\mathrm{S}}), -), \\
\Psi_{\mathrm{S}} \circ \TU' &\simeq \RHom_{Y'}(\TU'^{-1}(\Tilt_{\mathrm{S}}), -),
\end{align*}
and of objects
\begin{align*}
\ST_{\iota'_* \stsh_{\PP}(-3)}(\Tilt'_{\mathrm{S}}) &\simeq \stsh_{Y'} \oplus \stsh_{Y'}(-1) \oplus \stsh_{Y'}(-2) \oplus \ST_{\iota'_* \stsh_{\PP}(-3)}(\Sigma(-1)), \\
\UT(\Tilt_{\mathrm{S}}) &\simeq  \stsh_{Y'} \oplus \stsh_{Y'}(-1) \oplus \stsh_{Y'}(-2) \oplus \UT(\Sigma(-1)).
\end{align*}
Thus, it is enough to show that
\[ \ST_{j'_* \stsh_{\PP}(-3)}(\Sigma(-1)) \simeq \UT(\Sub(1)). \]
Applying the functor $\UT$ to the exact sequence
\[ 0 \to \Sub(1) \to V \otimes_{\CC} \stsh_Y(1) \to \Sub(2) \to 0 \]
gives an exact triangle on $\D(Y')$
\[ \UT(\Sub(1)) \to V \otimes_{\CC} \stsh_{Y'}(-1) \to \Sigma(-2) \to \UT(\Sub(1))[1]. \]
On the other hand,  the exact sequence
\[ 0 \to \stsh_{Y'}(-2) \to \Sigma(-1) \to \stsh_{Y'}(1) \to 0, \]
shows
\begin{align*}
\RHom_{Y'}(\iota'_* \stsh_{\PP}(-3), \Sigma(-1)) &\simeq \RHom_{Y'}(\iota'_* \stsh_{\PP}(-3), \stsh_{Y'}(1)) \\
&\simeq \RHom_{\PP}(\stsh_{\PP}(-3), \stsh_{Y'}(1) \otimes \omega_{\PP})[-2] \\
&\simeq \CC[-2].
\end{align*}
The non-trivial extension that corresponds to a generator of $\Ext^2_{Y'}(\iota'_* \stsh_{\PP}(-3), \Sigma(-1))$ is the one that was given in Lemma \ref{key lem ST}.
Thus the object $\ST_{\iota'_* \stsh_{\PP}(-3)}(\Sigma(-1))$ defined by a triangle
\[ \iota_*'\stsh_{\PP}(-3)[-2] \to \Sigma(-1) \to \ST_{\iota'_* \stsh_{\PP}(-3)}(\Sigma(-1)) \]
is quasi-isomorphic to a complex
\[ ( \cdots \to 0 \to 0 \to V \otimes_{\CC} \stsh_{Y'}(-1) \to \Sigma(-2) \to 0 \to 0 \to \cdots) \] 
whose degree zero part is $V \otimes_{\CC} \stsh_{Y'}(-1)$.
Hence there is an exact triangle
\[ \ST_{\iota'_* \stsh_{\PP}(-3)}(\Sigma(-1)) \to V \otimes_{\CC} \stsh_{Y'}(-1) \to \Sigma(-2) \to \ST_{\iota'_* \stsh_{\PP}(-3)}(\Sigma(-1))[1], \]
which implies the desired isomorphism $\UT(\Sub(1)) \simeq \ST_{\iota'_* \stsh_{\PP}(-3)}(\Sigma(-1))$.
\end{proof}

\subsection{The kernel of the equivalence $\TU'$}

In the same way as in Theorem \ref{thm flop-flop}, we can prove a ``flop-flop=twist" result on $Y$.
However, to prove this, we need the geometric description of the equivalence $\TU'$.
In the present subsection, we provide a Fourier-Mukai kernel of the equivalence $\TU'$.

\begin{lem} \label{key lem ST2}
There is an exact sequence
\[ 0 \to \stsh_Y(3) \to \Sub(2) \xrightarrow{b} \stsh_Y \to \stsh_{\LGr} \to 0 \]
on $Y$.
\end{lem}

\begin{proof}
On $Y'$, there is an exact sequence
\[ 0 \to \stsh_{Y'}(-3) \to \Sigma(-2) \to \stsh_{Y'} \to 0. \]
Restricting this to $Y^o$ and then extending to $Y$, we have a left exact sequence
\[ 0 \to \stsh_Y(3) \to \Sub(2) \xrightarrow{b} \stsh_Y.  \]
Thus, it is enough to show that $\Cok(b) \simeq \stsh_{\LGr}$.

Note that this sequence cannot be right exact.
Indeed, if this is a right exact sequence, then the sequence is locally split.
This contradicts to the fact that there is no non-trivial morphism from $\Sub(2)$ to $\stsh_{\LGr}$ on $\LGr$.

Let $j : Y^o \hookrightarrow Y$ be the open immersion.
As in the proof of Lemma \ref{key lem ST}, we have
\begin{align*}
\RR^1 j_* \stsh_Y(3)|_{Y^o} &\simeq \stsh_{\LGr}(3) \otimes  \bigoplus_{d \geq 1} \Sym^{d-2}(\Sub(-1)) \otimes \omega_{\LGr}.
\end{align*}
In particular, $R^1 j_*\stsh_Y(3)|_{Y^o}$ is a vector bundle on the zero section $\LGr$,
and hence its subsheaf $\Cok(b)$ is a torsion free sheaf on $\LGr$.
In particular, the surjective morphism $\stsh_Y \to \Cok(b)$ factors through a morphism $\stsh_{\LGr} \to \Cok(b)$, which is also surjective.
Since the sheaf $\Cok(b)$ is torsion free sheaf on $\LGr$, the surjective morphism $\stsh_{\LGr} \to \Cok(b)$ should be an isomorphism.
\end{proof}

%Let $\widetilde{Y}$ be a blowing up of $Y$ along the zero section $\LGr$ (or equivalently, of $Y'$ along the zero section $\PP$).
The exceptional divisor $E \subset \widetilde{Y}$ is isomorphic to $\PP_{\LGr}(\Sub(-1))$ and can be embedded into $\LGr \times \PP$ via the injective bundle map $\Sub(-1) \hookrightarrow V \otimes_{\CC} \stsh_{\LGr}(-1)$.
Put
\[ \widehat{Y} := \widetilde{Y} \cup_E (\LGr \times \PP). \]

\begin{thm} \label{FM ker}
The Fourier-Mukai kernel of the equivalence
\[ \TU' : \D(Y') \to \D(Y) \]
is given by the structure sheaf of $\widehat{Y}$:
\[ \TU' \simeq \FM_{\stsh_{\widehat{Y}}}^{Y' \to Y}. \]
\end{thm}

\begin{proof}
Let $\FM_{\stsh_{\widehat{Y}}} : \D(Y') \to \D(Y)$ be the Fourier-Mukai functor whose kernel is $\stsh_{\widehat{Y}}$,
and $\FM_{\stsh_{\widehat{Y}}}^!$ the right adjoint functor of $\FM_{\stsh_{\widehat{Y}}}$.

First let us show that $\FM_{\stsh_{\widehat{Y}}}(\Tilt'_{\mathrm{T}}) \simeq \Tilt_{\mathrm{U}}$.
Computations using the exact sequence
\[ 0 \to \stsh_{\widehat{Y}} \to \stsh_{\widetilde{Y}} \oplus \stsh_{\LGr \times \PP} \to \stsh_E \to 0 \]
yields
\[ \FM_{\stsh_{\widehat{Y}}}(\stsh_{Y'}(-a)) \simeq \stsh_Y(a) \]
for $0 \leq a \leq 2$ and
\begin{align*}
\mathcal{H}^i(\FM_{\stsh_{\widehat{Y}}}(\stsh_{Y'}(-3))) \simeq \begin{cases}
\stsh_Y(3) & \text{ if $i = 0$} \\
\stsh_{\LGr} & \text{ if $i = 1$} \\
0 & \text{ otherwise.}
\end{cases}
\end{align*}
It remains to show $\FM_{\stsh_{\widehat{Y}}}(\Sigma(-2)) \simeq \Sub(2)$.
Consider the exact sequence
\[ 0 \to \stsh_{Y'}(-3) \to \Sigma(-2) \to \stsh_{Y'} \to 0. \]
Applying the functor $\FM_{\stsh_{\widehat{Y}}}$ to this sequence and taking the cohomology long exact sequence show
that $\FM_{\stsh_{\widehat{Y}}}(\Sigma(-2))$ is a coherent sheaf on $Y$ that lies in a sequence
\[ 0 \to \stsh_Y(3) \to \FM_{\stsh_{\widehat{Y}}}(\Sigma(-2)) \to \stsh_Y \to \stsh_{\LGr} \to 0. \]
Since $\Ext^1_Y(I_{\LGr/Y}, \stsh_Y(3)) \simeq \Ext^2_Y(\stsh_{\LGr}, \stsh_Y(3)) \simeq \mathbb{C}$, this exact sequence coincides with the one given in Lemma \ref{key lem ST2}.
Therefore, $\FM_{\stsh_{\widehat{Y}}}(\Sigma(-2)) \simeq \Sub(2)$ as desired.

Now, for any object $x \in \D(Y)$, there are functorial isomorphisms
\begin{align*}
\Psi_{\mathrm{T}}' \circ \FM_{\stsh_{\widehat{Y}}}^!(x) = \RHom_{Y'}(\Tilt'_{\mathrm{T}}, \FM_{\stsh_{\widehat{Y}}}^!(x)) 
\simeq \RHom_Y(\FM_{\stsh_{\widehat{Y}}}(\Tilt_{\mathrm{T}}'), x)
\simeq  \RHom_Y(\Tilt_{\mathrm{U}}, x) = \Psi_U(x),
\end{align*}
of complexes over $X = \Spec R$.
This isomorphism together with the $R$-algebra homomorphism
\[ f \colon \End_Y(\Tilt_{\mathrm{U}}) \to \End_{Y'}(\Tilt'_{\mathrm{T}}) \]
associated with $\FM_{\stsh_{\widehat{Y}}}$
induces an action of $\Lambda_{\mathrm{U}} =  \End_Y(\Tilt_{\mathrm{U}})$ on $\Psi_{\mathrm{T}}' \circ \FM_{\stsh_{\widehat{Y}}}^!(x)$.
Since the the functor $\FM_{\stsh_{\widehat{Y}}}$ is an identity outside the singularity of $X$ by construction, 
the $R$-algebra homomorphism $f$ is also an identity on $X_{\mathrm{sm}}$,
and then the Cohen-Macaulay property implies that $f$ is an isomorphism that coincides with the identification in Lemma \ref{identification}.
Thus the above isomorphism $\Psi_{\mathrm{T}}' \circ \FM_{\stsh_{\widehat{Y}}}^!(x) \simeq \Psi_U(x)$ is $\Lambda_{\mathrm{U}}$-linear, and thus the diagram
\[ \begin{tikzcd}
\D(Y) \arrow[rr, "\FM_{\stsh_{\widehat{Y}}}^!"] \arrow[rd, "\Psi_{\mathrm{U}}"'] & & \D(Y') \arrow[ld, "\Psi_{\mathrm{T}}'"] \\
& \D(\Lambda_{\mathrm{U}}) &
\end{tikzcd} \]
commutes, which proves the result.
\end{proof}

\subsection{On the side of $Y$}
Finally, we prove the following ``flop-flop=twist" result for $Y$.

\begin{thm} \label{thm flop-flop2}
\begin{enumerate}
\item[(1)] The universal subbundle $\iota_* \Sub|_{\LGr}$ on $\LGr$ is a spherical object in $\D(Y)$.
\item[(2)] Consider a spherical twist $\ST_{\iota_*(\Sub)[2]} \in \Auteq(\D(Y))$ around a sheaf $\iota_*(\Sub)[2] = \iota_*(\Sub|_{\LGr})[2]$ on $\LGr$.
Then, there exists a functor isomorphism
\[ \Seg' \circ \UT \simeq \ST_{\iota_*(\Sub)[2]} \in \Auteq(\D(Y)). \]
\end{enumerate}
\end{thm}

\begin{proof}
By Theorem \ref{thm flop-flop}, it is enough to show that
\[ \TU'(\stsh_{\PP}(-3)) \simeq \Sub|_{\LGr}[2]. \]
The proof of Theorem \ref{thm flop-flop} gives a distinguished triangle
\[ \Sigma(-1) \to \UT(\Sub(1)) \to \stsh_{\PP}(-3)[-1] \to \Sigma(-1)[1]. \]
Applying a functor $\TU'$ to this sequence gives
\[ \TU'(\Sigma(-1)) \to \Sub(1) \to \TU'(\stsh_{\PP}(-3))[-1] \to \TU'(\Sigma(-1))[1]. \]
Thus, we have to compute the object $\TU'(\Sigma(-1))$.
Consider the exact sequence
\[ 0 \to \stsh_{Y'}(-2) \to \Sigma(-1) \to \stsh_{Y'}(1) \to 0 \]
on $Y'$.
Applying the functor $\TU'$ to this sequence gives an exact triangle
\[ \stsh_Y(2) \to \TU'(\Sigma(-1)) \to \TU'(\stsh_{Y'}(1)) \to \stsh_Y(2)[1]. \]
Then,  a computation using Theorem \ref{FM ker} shows that $\TU'(\stsh_{Y'}(1))$ lies in the following triangle
\[ \TU'(\stsh_{Y'}(1)) \to I_{\LGr/Y}(-1) \oplus (V^* \otimes_{\CC} \stsh_{\LGr}) \to \Sub^*|_{\LGr} \to \TU'(\stsh_{Y'}(1))[1]. \]
In addition, the diagram
\[ \begin{tikzcd}
 & V \otimes_{\CC} \stsh_{\LGr} \arrow[r, equal] \arrow[d] & V \otimes_{\CC} \stsh_{\LGr} \arrow[d] \\
 \TU'(\stsh_{Y'}(1)) \arrow[d, equal] \arrow[r] & I_{\LGr/Y}(-1) \oplus (V^* \otimes_{\CC} \stsh_{\LGr}) \arrow[r] \arrow[d] & \Sub^*|_{\LGr} \arrow[d] \\
 \TU'(\stsh_{Y'}(1)) \arrow[r] & I_{\LGr/Y}(-1) \arrow[r] & \Sub|_{\LGr}[1],
\end{tikzcd} \]
shows that $\TU'(\stsh_{Y'}(1))$ lies in the following sequence
\[  \TU'(\stsh_{Y'}(1)) \to I_{\LGr/Y} \to \Sub|_{\LGr}[1] \to  \TU'(\stsh_{Y'}(1))[1]. \]
On the other hand, by Lemma \ref{key lem ST2} and the construction of morphisms,
we have the following morphism between exact triangles
\[ \begin{tikzcd}
\stsh_Y(2) \arrow[r] \arrow[d, equal] & \TU'(\Sigma(-1)) \arrow[d]  \arrow[r] & \TU'(\stsh_{Y'}(1)) \arrow[d] \arrow[r] & \stsh_Y(2)[1] \arrow[d, equal] \\
\stsh_Y(2) \arrow[r] & \Sub(1) \arrow[r] & I_{\LGr/Y}(-1) \arrow[r] & \stsh_Y(2)[1].
\end{tikzcd} \]
Summarising the above computations implies
\begin{align*}
\TU'(\stsh_{\PP}(-3)) &\simeq \Cone(\TU'(\Sigma(-1)) \to \Sub(1))[1] \\
&\simeq \Cone(\TU'(\stsh_{Y'}(1)) \to I_{\LGr/Y})[1] \\
&\simeq \Sub|_{\LGr}[2].
\end{align*}
\end{proof}

\subsection{Another Flop-Flop=twist result}
Put
\begin{align*}
&\Tilt_{\mathrm{U}, 1} := \stsh_Y(-1) \oplus \stsh_Y \oplus \stsh_Y(1) \oplus \Sub(1), \\
&\Tilt'_{\mathrm{T}, 1} := \stsh_{Y'}(1) \oplus \stsh_{Y'} \oplus \stsh_{Y'}(-1) \oplus \Sigma(-1), \\
&\Lambda_{\mathrm{U}, 1} := \End_Y(\Tilt_{\mathrm{U}, 1}) = \End_{Y'}(\Tilt'_{\mathrm{T}, 1}).
\end{align*}
Note that $\Tilt_{\mathrm{U}, 1}$ was denoted by $\nu^2 \Tilt$ in Theorem \ref{mutation1}.
Consider derived equivalences
\begin{align*}
\Psi_{\mathrm{U}, 1} &:= \RHom_Y(\Tilt_{\mathrm{U}, 1},-) : \D(Y) \xrightarrow{\sim} \D(\modu \Lambda_{\mathrm{U}, 1}), \\
\Psi'_{\mathrm{T}, 1} &:= \RHom_{Y'}(\Tilt'_{\mathrm{T}, 1}, -) : \D(Y') \xrightarrow{\sim} \D(\modu \Lambda_{\mathrm{U}, 1}), \\
\UT_1 &:= (\Psi'_{\mathrm{T}, 1})^{-1} \circ \Psi_{\mathrm{U}, 1} : \D(Y) \xrightarrow{\sim} \D(Y'), \\
\TU'_1 &:= \Psi_{\mathrm{U}, 1}^{-1} \circ \Psi'_{\mathrm{T}, 1} : \D(Y') \xrightarrow{\sim} \D(Y).
\end{align*}
Then $\UT_1^{-1} \simeq \TU_1'$ and the following diagram commutes
\[ \begin{tikzcd}
\D(Y) \arrow[r, "\UT_1"] \arrow[d, "- \otimes \stsh_Y(1)"'] & \D(Y') \arrow[d, "- \otimes \stsh_{Y'}(-1)"] \\
\D(Y) \arrow[r, "\UT"] &D(Y').
\end{tikzcd} \]

\begin{thm} \label{FFT}
There is a functor isomorphism
\[ \TU_1' \circ \Seg \simeq \ST_{\stsh_{\LGr}(-1)} \in \Auteq(\D(Y)). \]
\end{thm}

\begin{proof}
We have to show the following diagram commutes:
\[ \begin{tikzcd}
\D(Y) \arrow[r, "\Seg"] \arrow[d, "\ST_{\stsh_{\LGr}(-1)}"'] & \D(Y') \arrow[d, "\Psi'_{\mathrm{T}, 1}"] \\
\D(Y) \arrow[r, "\Psi_{\mathrm{U}, 1}"] & \D(\modu \Lambda_{\mathrm{U}, 1}).
\end{tikzcd} \]
As in the proof of Theorem \ref{thm flop-flop}, it is enough to show that
\[ \Seg'(\stsh_{Y'}(1)) \simeq \ST_{\stsh_{\LGr}(-1)}^{-1}(\stsh_Y(-1)). \]
First, using an exact sequence
\[ 0 \to I_{\LGr/Y}(-1) \to \stsh_Y(-1) \to \stsh_{\LGr}(-1) \to 0 \]
and a computation
\[ \RHom_Y(\iota_* \stsh_{\LGr}(-1), \stsh_Y(-1)) \simeq \RGamma(\LGr, \stsh_{\LGr}(-3))[-2] \simeq \CC[-5] \]
imply $\RHom_Y(\iota_* \stsh_{\LGr}(-1), I_{\LGr/Y}(-1)) \simeq \CC[1]$,
and hence $\ST_{\stsh_{\LGr}(-1)}(I_{\LGr/Y}(-1)) = \stsh_Y(-1)$.
On the other hand, applying the functor $\Seg'$ to the sequence
\[ 0 \to \stsh_{Y'}(-2) \to \Sigma(-1) \to \stsh_{Y'}(1) \to 0, \]
gives a triangle
\[ \stsh_Y(2) \to \Sub(1) \to \Seg'(\stsh_{Y'}(1)) \to \stsh_Y(2)[1]. \]
Then Lemma \ref{key lem ST2} implies $\Seg'(\stsh_{Y'}(1)) \simeq I_{\LGr/Y}(-1)$ as desired.
\end{proof}

\subsection{Multi-mutation=twist result}

Note that $\Lambda_{\mathrm{U}, 1}$ is the endomorphism ring of an $R$-module
\[ M_{-1} \oplus M_0 \oplus M_1 \oplus S_1. \]
Let $W' := M_{0} \oplus M_1 \oplus S_1$.
This $W'$ was denoted by $W_3$ in Theorem \ref{mutation1}.
Recall that $\Lambda_{\mathrm{S}}$ is the endomorphism ring of $W' \oplus M_{2}$.

\begin{prop} \label{mutation3}
Tthe following two isomorphism of $R$-modules holds.
\begin{enumerate}
\item[(1)] $\mu_{W'}\mu_{W'}\mu_{W'}(W' \oplus M_{-1}) \simeq W' \oplus M_{2}$.
\item[(2)] $\mu_{W'}(W' \oplus M_{2}) \simeq W' \oplus M_{-1}$.
\end{enumerate}
Moreover, the induced IW functor
\[ \Phi_{W'} : \D(\modu \Lambda_{\mathrm{S}}) \to \D(\modu \Lambda_{\mathrm{U},1}) \]
from (2) is isomorphic to $\Psi'_{\mathrm{U},1} \circ (\Psi'_{\mathrm{S}})^{-1}$.
\end{prop}

\begin{proof}
(1) was proved in Theorem \ref{mutation1}.
One can show (2) by using Lemma \ref{lem approx2}.
The exchange sequence for (2) is given by the dual of the exact sequence
\[ 0 \to M_{2} \to S_1 \to M_{-1} \to 0. \]
This sequence is obtained by taking the global section of the sequence
\[ 0 \to \stsh_{Y'}(-2) \to \Sigma(-1) \to \stsh_{Y'}(1) \to 0. \]
\end{proof}

The following is a ``multi-mutation=twist" result for the Abuaf flop.

\begin{thm} \label{multimutationtwist}
By Proposition \ref{mutation3}, we have an autoequivalence of $\D(\modu \Lambda_{\mathrm{U},1})$ by composing four IW mutation functors at $W'$:
\[ \Phi_{\mathrm{W'}} \circ \Phi_{\mathrm{W'}} \circ \Phi_{\mathrm{W'}} \circ \Phi_{\mathrm{W'}} \in \Auteq(\D(\modu \Lambda_{\mathrm{U},1})). \]
This autoequivalence corresponds to a spherical twist $\ST_{\stsh_{\LGr}(-1)} \in \Auteq(\D(Y))$ under the identification
\[ \Phi_{\mathrm{U}, 1} : \D(Y) \xrightarrow{\sim} \D(\modu \Lambda_{\mathrm{U}, 1}). \]
\end{thm}

\begin{proof}
By Theorem \ref{mutation1}, Theorem \ref{FFT}, and Proposition \ref{mutation3}, we have the following commutative diagram
\[ \begin{tikzcd}
\D(Y) \arrow[d, "\Psi_{\mathrm{U}, 1}"] \arrow[r, equal] & \D(Y) \arrow[d, "\Psi_{\mathrm{S}}"] \arrow[r, "\ST_{\stsh_{\LGr}(-1)}"] & \D(Y) \arrow[d, "\Psi_{\mathrm{U}, 1}"] \\
\D(\modu \Lambda_{\mathrm{U}, 1}) \arrow[r, "\Phi^3_{W'}"] & \D(\modu \Lambda_{\mathrm{S}}) \arrow[r, "\Phi_{W'}"] & \D(\modu \Lambda_{\mathrm{U}, 1}) \\
&\D(Y') \arrow[r, equal] \arrow[u, "\Psi_{\mathrm{S}}"] & \D(Y'), \arrow[u, "\Psi_{\mathrm{T}, 1}"]
\end{tikzcd} \]
and the result follows from this diagram.
\end{proof}

\begin{rem} \rm
Compare this result with \cite[Theorem 5.18 and Remark 5.19]{H17a}.
There, the author proved that a P-twist on the cotangent bundle $T^*\PP^n$ of $\PP^n$ associated to the sheaf $\stsh_{\PP}(-1)$ on the zero section $\PP \subset T^*\PP^n$ corresponds to a composition of $2n$ IW mutations of an NCCR.
\end{rem}
\begin{rem} \rm
By Theorem \ref{multimutationtwist}, we notice that an autoequivalence
\[ \Phi_{\mathrm{W'}} \circ \Phi_{\mathrm{W'}} \circ \Phi_{\mathrm{W'}} \circ \Phi_{\mathrm{W'}} \in \Auteq(\D(\modu \Lambda_{\mathrm{U},1})) \]
corresponds to a spherical twist
\[ \ST_{\sh} \in \Auteq(\D(Y')) \]
on $Y'$ around an object $\sh := \UT_1(\stsh_{\LGr}(-1))$,
under the identification
\[ \Psi'_{\mathrm{T}, 1} : \D(Y') \xrightarrow{\sim} \D(\modu \Lambda_{\mathrm{U}, 1}). \]
Note that $\sh$ is also a spherical object on $Y'$ because $Y'$ has a trivial canonical bundle.
However, in contrast to the case for $Y$, the object $\sh$ is not contained in the subcategory $\iota'_*\D(\PP)$ of $\D(Y')$.

Indeed, we have
\begin{align*}
&\RHom_{Y'}(\sh, \stsh_{Y'} \oplus \stsh_{Y'}(-1) \oplus \Sigma(-1)) \\
\simeq &\RHom_Y(\iota_*\stsh_{\LGr}(-1), \stsh_Y \oplus \stsh_Y(1) \oplus \Sub(1)) \\
\simeq &\RHom_{\LGr}(\stsh_{\LGr}(-1), (\stsh_{\LGr} \oplus \stsh_{\LGr}(1) \oplus \Sub(1)) \otimes \omega_{\LGr})[-2] \\
=  &0.
\end{align*}
Thus, if $\sh \simeq \iota'_* F$ for some $F \in \D(\PP)$, then
\[ 0 = \RHom_{Y'}(\iota'_* F, \stsh_{Y'} \oplus \stsh_{Y'}(-1) \oplus \Sigma(-1)) \simeq \RHom_{\PP}(F, \bigoplus_{k=-2}^1 \stsh_{\PP}(k) \otimes \omega_{\PP})[-2]. \]
Since the object $\bigoplus_{k=-2}^1 \stsh_{\PP}(k) \otimes \omega_{\PP}$ spans the derived category $\D(\PP)$ of the three dimensional projective space, the above vanishing implies $F = 0$.
This is a contradiction.
\end{rem}

\section{Mutation of exceptional objects} \label{sect exc colle}

In Section \ref{sect: mutation}, we construct a resolution of a sheaf by using the mutation of exceptional objects.
In this section, we recall the definition of exceptional objects and mutation of them, and explain how to find resolutions that we used in Section \ref{sect: mutation}.

\subsection{Definition}

Let $\mathcal{D}$ be a triangulated category with finite dimensional $\Hom$ spaces.

\begin{defi}
\rm
\begin{enumerate}
\item[(i)] An object $\mathcal{E} \in \mathcal{D}$ is called an \textit{exceptional object} if 
\begin{equation*}
\Hom_{\mathcal{D}}(\mathcal{E}, \mathcal{E}[i]) = \left\{
\begin{array}{ll}
\CC  & \mbox{if $i = 0$,} \\
0  & \mbox{if $i \neq 0$.}
\end{array}
\right.
\end{equation*}
\item[(ii)] A sequence of exceptional objects $\mathcal{E}_1, \dots, \mathcal{E}_r$ is called an \textit{exceptional collection} if $\RHom_{\mathcal{D}}(\mathcal{E}_l, \mathcal{E}_k) = 0$ for all $1 \leq k < l \leq r$.
\item[(iii)] An exceptional collection $\mathcal{E}_1, \dots, \mathcal{E}_r$ is \textit{full} if it generates the whole category $\mathcal{D}$. In such case, we write
\[ \mathcal{D} = \langle \mathcal{E}_1, \dots, \mathcal{E}_r \rangle. \]
\end{enumerate}
\end{defi}

\begin{eg}[\cite{Bei79}, \cite{Kuz08}]
\rm
\begin{enumerate}
\item[(1)] An $n$-dimensional projective space $\mathbb{P}^n$ has a full exceptional collection consisting of line bundles called the Beilinson collection
\[  \D(\mathbb{P}^n) = \langle \stsh, \stsh(1), \stsh(2), \dots, \stsh(n) \rangle. \]
\item[(2)] Let $V$ be a four dimensional symplectic vector space and $\LGr(V)$ the Lagrangian Grassmannian of $V$.
Kuznetsov found a full exceptional collection
\[ \D(\LGr(V)) = \langle \stsh_{\LGr}, \Sub(1), \stsh_{\LGr}(1), \stsh_{\LGr}(2) \rangle. \]
\end{enumerate}
\end{eg}

For an object $\mathcal{E} \in \mathcal{D}$, we define subcategories $\mathcal{E}^{\bot}, {}^{\bot}\mathcal{E} \subset \mathcal{D}$ by
\begin{align*}
\mathcal{E}^{\bot} &:= \{ \mathcal{F} \in \mathcal{D} \mid \RHom_{\mathcal{D}}(\mathcal{E}, \mathcal{F}) = 0 \} \\
{}^{\bot}\mathcal{E} &:= \{ \mathcal{F} \in \mathcal{D} \mid \RHom_{\mathcal{D}}(\mathcal{F}, \mathcal{E}) = 0 \}.
\end{align*}

The following lemma is useful.

\begin{lem} \label{useful}
Let
\[ \mathcal{D} = \langle \mathcal{E}_1, \dots, \mathcal{E}_r \rangle =  \langle \mathcal{E}'_1, \dots, \mathcal{E}'_r \rangle \]
be two full exceptional collections with the same length.
Let $1 \leq i \leq r$ and assume that $\mathcal{E}_j = \mathcal{E}'_j$ holds for all $j \neq i$.
Then, we have
\[ \mathcal{E}_i = \mathcal{E}'_i \]
up to shift.
\end{lem}

\begin{proof}
This Lemma follows from the fact
\[  {}^{\bot}\mathcal{E}_{1} \cap \cdots \cap {}^{\bot}\mathcal{E}_{i-1} \cap \mathcal{E}_{i+1}^{\bot} \cap \cdots \cap \mathcal{E}_{r}^{\bot}  = \D(\Spec \CC) \otimes_{\CC} \mathcal{E}_i. \]
For this fact, see \cite{Bondal90}.
\end{proof}

\begin{defi}
\rm
Let $\mathcal{E} \in \mathcal{D}$ be an exceptional object. For an object $\mathcal{F}$ in ${}^{\bot}\mathcal{E}$, we define the \textit{left mutation of $\mathcal{F}$ through $\mathcal{E}$} as the object $\mathbb{L}_{\mathcal{E}}(\mathcal{F})$ in $\mathcal{E}^{\bot}$ that lies in an exact triangle
\[ \RHom(\mathcal{E}, \mathcal{F}) \otimes \mathcal{E} \map \mathcal{F} \map \mathbb{L}_{\mathcal{E}}(\mathcal{F}). \]
Similarly, for an object $\mathcal{G}$ in $\mathcal{E}^{\bot}$, we define the \textit{right mutation of $\mathcal{G}$ through $\mathcal{E}$} as the object $\mathbb{R}_{\mathcal{E}}(\mathcal{G})$ in ${}^{\bot}\mathcal{E}$ which lies in an exact triangle
\[ \mathbb{R}_{\mathcal{E}}(\mathcal{G}) \map \mathcal{G} \map \RHom(\mathcal{G}, \mathcal{E})^* \otimes \mathcal{E}. \]
\end{defi}

\begin{lem}[\cite{Bondal90}] \label{lem exc obj mut}
Let $\mathcal{E}_1, \mathcal{E}_2$ be an exceptional pair (i.e.\ an exceptional collection consisting of two objects).
\begin{enumerate}
\item[(i)] The left (resp. right) mutated object $\mathbb{L}_{\mathcal{E}_1}(\mathcal{E}_2)$ (resp. $\mathbb{R}_{\mathcal{E}_2}(\mathcal{E}_1)$) is again an exceptional object.
\item[(ii)] The pairs of exceptional objects $\mathcal{E}_1, \mathbb{R}_{\mathcal{E}_1}(\mathcal{E}_2)$ and $\mathbb{L}_{\mathcal{E}_2}(\mathcal{E}_1), \mathcal{E}_2$ are again exceptional pairs.
\end{enumerate}
Let $\mathcal{E}_1, \dots, \mathcal{E}_r$ be a full exceptional collection in $\mathcal{D}$. Then
\begin{enumerate}
\item[(iii)] The collection 
\[ \mathcal{E}_1, \dots, \mathcal{E}_{i-1}, \mathbb{L}_{\mathcal{E}_i}(\mathcal{E}_{i+1}), \mathcal{E}_i, \mathcal{E}_{i+2}, \dots, \mathcal{E}_r \]
is again full exceptional for each $1 \leq i \leq r-1$. Similarly, the collection 
\[ \mathcal{E}_1, \dots, \mathcal{E}_{i-2}, \mathcal{E}_i, \mathbb{R}_{\mathcal{E}_i}(\mathcal{E}_{i-1}), \mathcal{E}_{i+1}, \dots, \mathcal{E}_r \] 
is again full exceptional for each $2 \leq i \leq r$.
\item[(iv)] Assume in addition that the category $\mathcal{D}$ admits the Serre functor $S_{\mathcal{D}}$.
Then the following collections
\begin{align*}
\mathcal{E}_2, \dots, \mathcal{E}_{r-1}, \mathcal{E}_r, S_{\mathcal{D}}^{-1}(\mathcal{E}_1) ~~ \text{~ and ~} ~~ S_{\mathcal{D}}(\mathcal{E}_r), \mathcal{E}_1, \mathcal{E}_2, \dots, \mathcal{E}_{r-1}
\end{align*}
are full exceptional collections on $\mathcal{D}$.
\end{enumerate}
\end{lem}

\begin{eg} \rm
By taking mutations, we have the following different full exceptional collections for $\D(\LGr(V))$:
\begin{align*}
\D(\LGr(V)) &= \langle \stsh_{\LGr}, \Sub(1), \stsh_{\LGr}(1), \stsh_{\LGr}(2) \rangle \\
&= \langle \stsh_{\LGr}, \stsh_{\LGr}(1), \Sub(2), \stsh_{\LGr}(2) \rangle \\
&= \langle \stsh_{\LGr}, \stsh_{\LGr}(1), \stsh_{\LGr}(2), \Sub(3) \rangle \\
&= \langle \Sub, \stsh_{\LGr}, \stsh_{\LGr}(1), \stsh_{\LGr}(2) \rangle.
\end{align*}
\end{eg}

\subsection{Application for finding resolutions}

\begin{lem} \label{app lem exact}
There exist the following exact sequences on $\LGr(V)$
\begin{enumerate}
\item $0 \to \stsh_{\LGr}(-3) \to \Sub(-2)^{\oplus 4} \to \stsh_{\LGr}(-2)^{\oplus 11} \to \stsh_{\LGr}(-1)^{\oplus 5} \to \stsh_{\LGr} \to 0$.
\item $0 \to \stsh_{\LGr}(-2) \to \stsh_{\LGr}(-1)^{\oplus 5} \to \stsh_{\LGr}^{\oplus 11} \to \Sub(1)^{\oplus 4} \to \stsh_{\LGr}(1) \to 0.$
\item $0 \to \stsh_{\LGr}(-1) \to \stsh_{\LGr}^{\oplus 5} \to \Sub(1)^{\oplus 4} \to \stsh_{\LGr}(1)^{\oplus 5} \to \stsh_{\LGr}(2) \to 0$.
\end{enumerate}
\end{lem}

\begin{proof}
Let us consider a full exceptional collection
\[ \D(\LGr(V)) = \langle \stsh_{\LGr}(-3), \Sub(-2), \stsh_{\LGr}(-2), \stsh_Y(-1) \rangle. \]
To prove (1), we mutate $\stsh_{\LGr}(-3)$ over an exceptional collection $\Sub(-2), \stsh_{\LGr}(-2), \stsh_Y(-1)$.
By Lemma \ref{lem exc obj mut} and Lemma \ref{useful}, there is an isomorphism
\[ \RRR_{\Sub(-2)}(\stsh_{\LGr}(-3)) \simeq \LLL_{\stsh_{\LGr}(-2)}\LLL_{\stsh_{\LGr}(-1)}(\stsh_{\LGr}(-3) \otimes \omega_{\LGr}^{-1}) \]
up to shift.
Note that $\stsh_{\LGr}(-3) \otimes \omega_{\LGr}^{-1} \simeq \stsh_{\LGr}$.

First, we have
\[ \RHom_{\LGr(V)}(\stsh_{\LGr}(-3), \Sub(-2)) \simeq \CC^{4} \]
and hence the object $\RRR_{\Sub(-2)}(\stsh_{\LGr}(-3))[1]$ lies on an exact triangle
\[ \stsh_{\LGr}(-3) \xrightarrow{\text{ev}} \Sub(-2)^{\oplus 4} \to \RRR_{\Sub(-2)}(\stsh_{\LGr}(-3))[1] \to \stsh_{\LGr}(-3)[1]. \]
Since $\stsh_{\LGr}(-3)$ and $\Sub(-2)^{\oplus 4}$ are vector bundles on $\LGr(V)$, the map $\text{ev}$ should be injective and hence the object
$\RRR_{\Sub(-2)}(\stsh_{\LGr}(-3))[1]$ is a sheaf on $\LGr(V)$.
Thus, we put
\[ \sh := \RRR_{\Sub(-2)}(\stsh_{\LGr}(-3))[1]. \]

Next, we have $\RHom_{\LGr(V)}(\stsh_{\LGr}(-1), \stsh_{\LGr}) \simeq \CC^5$ and hence
\[ \LLL_{\stsh_{\LGr}(-1)}(\stsh_{\LGr})[-1] \simeq \Omega_{\PP^4}^1|_{\LGr}. \]
Moreover, an easy computation shows that $\RHom_{\LGr(V)}(\stsh_{\LGr}(-2), \Omega_{\PP^4}^1|_{\LGr}) \simeq \CC^{11}$ and hence the object
$\LLL_{\stsh_{\LGr}(-2)}(\Omega_{\PP^4}^1|_{\LGr})$ lies on the exact sequence
\[ \stsh_{\LGr}(-2)^{\oplus 11} \xrightarrow{\text{ev}} \Omega_{\PP^4}^1|_{\LGr} \to \LLL_{\stsh_{\LGr}(-2)}(\Omega_{\PP^4}^1|_{\LGr}) \to \stsh_{\LGr}(-2)^{\oplus 11}[1]. \]
From the above computation, the object $\LLL_{\stsh_{\LGr}(-2)}(\Omega_{\PP^4}^1|_{\LGr})$ should be a sheaf on $\LGr(V)$ (up to shift)
whose generic rank is equal to $7$.
Thus the map $\text{ev}$ is surjective and $\LLL_{\stsh_{\LGr}(-2)}(\Omega_{\PP^4}^1|_{\LGr})[-1] \simeq \sh$.

Summarising the above arguments, we have three exact sequences:
\begin{align*}
&0 \to \stsh_{\LGr}(-3) \to \Sub(-2)^{\oplus 4} \to \sh \to 0, \\
&0 \to \sh \to \stsh_{\LGr}(-2)^{\oplus 11} \to \Omega_{\PP^4}^1|_{\LGr} \to 0, ~ \text{and} \\
&0 \to \Omega_{\PP^4}^1|_{\LGr} \to \stsh_{\LGr}(-1)^{\oplus 5} \to \stsh_{\LGr} \to 0.
\end{align*}
Combining these three exact sequences gives the exact sequence (1).

To obtain (2), take the dual of (1), and then apply $\otimes \stsh(-2)$.

The sequence (3) is proved using a similar argument as in (1).
This exact sequence comes from the right mutation of $\stsh_{\LGr}(-1)$ over an exceptional collection $\stsh_{\LGr}, \Sub(1), \stsh_{\LGr}(1)$.
\end{proof}

%By using similar arguments, we can obtain the long exact sequences that we used in the proof of Theorem \ref{mutation1}.


\begin{thebibliography}{99}
\bibitem{ADM15} N. Addington, W. Donovan, C. Meachan, \textit{Mukai flops and P-twists}, J. Reine Angew. Math. 748 (2019), 227--240.

\bibitem{BB15} A. Bodzenta, A. Bondal, \textit{Flops and spherical functors}, Compos. Math. 158 (2022), no. 5, 1125--1187. 

\bibitem{Bei79} A. Beilinson, \textit{Coherent sheaves on $\mathbb{P}^n$ and problems in linear algebra}, Funct. Anal. Appl. \textbf{12}(3) (1978), 68--69.

\bibitem{Bondal90} A. Bondal, \textit{Representation of associative algebras and coherent sheaves}, Math. USSR Izv. \textbf{34}(1) (1990), 23--42.

\bibitem{Ca12} S. Cautis, \textit{Flops and about: a guide}, Derived categories in algebraic geometry, 61--101, EMS Ser. Congr. Rep., Eur. Math. Soc., Z\"{u}rich, 2012.

\bibitem{DW15} W. Donovan, M. Wemyss, \textit{Twists and braids for general 3-fold flops}, J. Eur. Math. Soc. (JEMS) 21 (2019), no. 6, 1641--1701.

\bibitem{DW16} W. Donovan, M. Wemyss, \textit{Noncommutative deformations and flops}, Duke Math. J., \textbf{165}(8) (2016), 1397--1474.

\bibitem{FH91} W. Fulton, J. Harris, \textit{Representation theory.  A first course.}, GTM 129, 1991, Springer.

\bibitem{H17a} W. Hara, \textit{Non-commutative crepant resolution of minimal nilpotent orbit closures of type A and Mukai flops}, Adv. Math., \textbf{318} (2017), 355--410.

\bibitem{HV07} L. Hille, M. Van den Bergh, \textit{Fourier-Mukai transforms}, In: Handbook of Tilting Theory. London Math. Soc. Lecture Note Ser., vol.332, Cambridge University Press, Cambridge (2007), 147--177..

\bibitem{IR08} O. Iyama, I. Reiten.: \textit{Fomin-Zelevinsky mutation and tilting modules over Calabi-Yau algebras}, Am. J. Math. \textbf{130}(4) (2008), 1087--1149.

\bibitem{IW14} O. Iyama, M. Wemyss, \textit{Maximal modifications and {A}uslander-{R}eiten duality for non-isolated singularities}, Invent. Math., \textbf{197}(3) (2014), 521--586.

\bibitem{Ka02} Y. Kawamata, \textit{{$D$}-equivalence and {$K$}-equivalence}, J. Differential Geom., \textbf{61}(1) (2002), 147--171.

%\bibitem[KPS17]{KPS17} A. Krug, D. Ploog, P. Sosna, \textit{Derived categories of resolutions of cyclic quotient singularities}, \url{https://arxiv.org/abs/1701.01331}.

\bibitem{Kuz08} A. Kuznetsov, \textit{Exceptional collections for Grassmannians of isotropic lines}, Proc. London Math. Soc. \textbf{97}
(2008), 155--182.

\bibitem{Li17} D. Li, \textit{On certain K-equivalent birational maps}, Math. Z. 291 (2019), no. 3-4, 959--969. 

\bibitem{Nak16} Y. Nakajima, \textit{Mutations of splitting maximal modifying modules: The case of reflexive polygons},  Int. Math. Res. Not. IMRN 2019, no. 2, 470--550.

\bibitem{Na03} Y. Namikawa, \textit{Mukai flops and derived categories}, J. Reine Angew. Math., \textbf{560} (2003), 65--76.

\bibitem{Seg16} E. Segal, \textit{A new 5-fold flop and derived equivalence}, Bull. Lond. Math. Soc., \textbf{48}(3) (2016), 533--538.

\bibitem{ST01} P. Seidel, R. Thomas, \textit{Braid group actions on derived categories of coherent sheaves}, Duke. Math. J., \textbf{108} (2001), 37--108.

\bibitem{To07} Y. Toda, \textit{On certain generalization of spherical twists}, Bull. Soc. Math. France. \textbf{135}(1) (2007), 119--134.

\bibitem{TU10} Y. Toda, H. Uehara, \textit{Tilting generators via ample line bundles}, Adv. Math., \textbf{223}(1) (2010), 1--29.

\bibitem{VdB04a} M. Van den Bergh, \textit{Three-dimensional flops and noncommutative rings}, Duke Math. J., \textbf{122}(3) (2004), 423--455.

\bibitem{VdB04b} M. Van den Bergh, \textit{Non-commutative crepant resolutions}, The legacy of {N}iels {H}enrik {A}bel, 749--770, Springer, Berlin, 2004.

\bibitem{We14} M. Wemyss, \textit{Flops and Clusters in the Homological Minimal Model Program}, Invent. Math. 211 (2018), no. 2, 435--521.
\end{thebibliography}
\end{document}